\RequirePackage{snapshot}       
\documentclass[a4paper, 11pt]{article}

\usepackage[english]{babel}
\usepackage{microtype}
\usepackage{array,multirow} 
\usepackage{mathtools} 
\usepackage{amsthm} 
\usepackage{graphicx}
\usepackage{rotating}
\usepackage{adjustbox} 
\usepackage{tikz}
\usetikzlibrary{calc}
\usetikzlibrary{cd} 
\usetikzlibrary{shapes.geometric}
\usetikzlibrary{trees}
\usetikzlibrary{positioning}
\usetikzlibrary{arrows.meta}
\usepackage[alwaysadjust]{paralist}
\usepackage[%
breaklinks,
colorlinks=true,
linkcolor=black,
anchorcolor=black,
citecolor=black,
filecolor=black,
menucolor=black,
urlcolor=black]{hyperref}
\usepackage[font={small}]{caption}
\usepackage[charter]{mathdesign}
\usepackage[textwidth=365pt, textheight=594pt, vmarginratio=2:3]{geometry}

\usepackage[T3,T1]{fontenc}
\DeclareSymbolFont{tipa}{T3}{cmr}{m}{n}
\DeclareMathAccent{\invbreve}{\mathalpha}{tipa}{16}
\DeclareMathAccent{\mybreve}{\mathalpha}{tipa}{8}

\providecommand{\noopsort}[1]{}

\newcommand{\sminfty}{\makebox{\small$\infty$}}
\newcommand{\R}{\mathbb{R}}
\newcommand{\Q}{\mathbb{Q}}
\newcommand{\C}{\mathbb{C}}
\newcommand{\Z}{\mathbb{Z}}

\newcommand{\im}{\operatorname{Im}}
\newcommand{\trace}{\operatorname{trace}}
\newcommand{\length}{\operatorname{length}}
\newcommand{\group}[1]{\mathit{#1}}
\newcommand{\RP}{{\R}{P}}
\newcommand{\GLTZ}{\group{GL}_2(\Z)}
\newcommand{\PGLTZ}{\group{PGL}_2(\Z)}
\newcommand{\PSLTZ}{\group{PSL}_2(\Z)}
\newcommand{\SLTZ}{\group{SL}_2(\Z)}
\newcommand{\SLTR}{\group{SL}_2(\R)}
\newcommand{\PSLTR}{\group{PSL}_2(\R)}

\theoremstyle{plain}
\newtheorem{lemma}{Lemma}[section]
\newtheorem{theorem}[lemma]{Theorem}
\newtheorem{corollary}[lemma]{Corollary}
\newtheorem{conjecture}[lemma]{Conjecture}

\theoremstyle{definition}
\newtheorem{definition}[lemma]{Definition}

\newtheorem{question}[lemma]{Question}
\newtheorem{remark}[lemma]{Remark}

\newtheorem{examples}[lemma]{Examples}

\title{The worst approximable rational numbers}

\author{Boris Springborn}
\date{}

\begin{document}

\maketitle

\begin{abstract}
  We classify and enumerate all rational numbers with approximation
  constant at least $\frac{1}{3}$ using hyperbolic geometry. Rational
  numbers correspond to geodesics in the modular torus with both ends
  in the cusp, and the approximation constant measures how far they
  stay out of the cusp neighborhood in between. Compared to the
  original approach, the geometric point of view eliminates the need
  to discuss the intricate symbolic dynamics of continued fraction
  representations, and it clarifies the distinction between the two
  types of worst approximable rationals: (1) There is a plane forest
  of \emph{Markov fractions} whose denominators are Markov
  numbers. They correspond to simple geodesics in the modular torus
  with both ends in the cusp. (2) For each Markov fraction, there are
  two infinite sequences of \emph{companions}, which correspond to
  non-simple geodesics with both ends in the cusp that do not
  intersect a pair of disjoint simple geodesics, one with both ends in
  the cusp and one closed.
  
  \vspace{0.5\baselineskip}\noindent%
  \emph{MSC (2010).} 
  10F20, 30F60
  \\
  \emph{Key words and phrases.}
  Approximation constant,
  Diophantine approximation,
  Markov equation,
  modular torus
\end{abstract}

\section{Introduction}
\label{sec:intro}

For any real number $x$, the \emph{approximation constant}
\begin{equation}
  \label{eq:C}
  C(x)\;=\;\inf_{\frac{a}{b}\;\in\;\Q\,\setminus\,\{x\}}\;
  b^{2}\cdot\,\Big|\,x\;-\;\frac{a}{b}\,\Big|\;,
\end{equation}
measures how well $x$ can be approximated by rational
numbers~$\frac{a}{b}\not=x$. Excluding~$x$ from the admissible
approximants~$\frac{a}{b}$ makes $C(x)$ strictly positive for all
rational numbers $x$. (This is easy to see and will be discussed
shortly.) Intuitively, the larger the approximation constant $C(x)$ of
a rational number $x\in\Q$, the more isolated the rational number $x$
is among the other rational numbers.

An irrational number $x$ is sometimes called \emph{badly approximable}
if $C(x)>0$, and this is the case if and only if the sequence of
partial denominators $a_{k}$ of the continued fraction expansion
\begin{equation*}
  x\;=\;a_{0}
  + \cfrac{1}{a_{1}
    +\cfrac{1}{a_{2} + \ldots}}
\end{equation*}
is bounded (see, e.g.,~\cite[Prop.~1.32]{Aigner13}
or~\cite[{\S}10.8]{Hardy08}).

By contrast (as stated before), $C(x)>0$ for all rational
numbers $x=\frac{p}{q}$. Moreover, the nonzero infimum
in~\eqref{eq:C} is attained for some fraction~$\frac{a}{b}$ with
denominator~$b\leq q$. Because, on the one hand, if
$b>q$ and $\frac{a}{b}\not=\frac{p}{q}$, then
\begin{equation*}
  b^{2}\cdot\Big|\,\frac{p}{q}\;-\;\frac{a}{b}\,\Big|\;=\;
  \frac{b}{q}\cdot\big|\,pb\;-\;qa\,\big| > 1. 
\end{equation*}
On the other hand, if $b=1$ and $a$ is a closest integer to $x$
(except $x$ itself if $x\in\Z$) then
\begin{equation*}
  b^{2}\cdot\big|\,x-\frac{a}{b}\,\big| = |\,x-a\,| \leq 1
\end{equation*}
with equality if and only if $x\in\Z$.

This shows not only that all rational numbers are badly approximable
in the sense that $C(x)>0$, but also that the very worst approximable
rational numbers are the integers $x\in\Z$ with $C(x)=1$. They are
followed by the half-integers $x\in \Z+\frac{1}{2}$ with
$C(x)=\frac{1}{2}$, which are best approximated by the two nearest
integers. This can be seen as a striking confirmation of Hardy and
Wright's general observation: ``From the point of view of rational
approximation, \emph{the simplest numbers are the worst}''
\cite[p.~209, emphasis in original]{Hardy08}.

This article is about a geometric method to classify and enumerate the
rational numbers $x$ with approximation constant
$C(x)\geq\frac{1}{3}$. Previously, the real numbers with
$C(x)>\frac{1}{3}$ were classified in terms of their continued
fraction representations by Flahive~\cite{gbur78} for rational~$x$ and
by Gurwood~\cite{gurwood76} for irrational~$x$. Flahive's
classification apparently extends to include rational numbers with
$C(x)=\frac{1}{3}$.

All of this is closely related to the classification of \emph{Markov
  irrationals} and \emph{Markov forms}, i.e., of irrational numbers
$x$ with \emph{Lagrange number} $L(x)>\frac{1}{3}$ and of indefinite
quadratic forms $f(x,y)=Ax^{2}+2Bxy+Cy^{2}$ with \emph{Markov
  constant} $M(f)>\frac{2}{3}$, where
\begin{gather}
  \label{eq:Lagrange_number}
  L(x)=\liminf_{b\rightarrow\infty}\,\big(b^{2}\cdot\min_{a\in\Z}\,\big|x-\tfrac{a}{b}\big|\big)
  \\\intertext{and}
  \label{eq:Markov_constant}
  M(f)\;=\;
  \inf_{(a,b)\,\in\,\Z^{2}\setminus\{(0,0)\}}
  \frac{f(a,b)}{\sqrt{-\det\,f}}\,.
\end{gather}
The monographs~\cite{Aigner13,Cassels57,Cusick89} and their
bibliographies are excellent resources for that theory. (Often the
Lagrange number is defined as the reciprocal value, $1/L(x)$.)

While the Lagrange number $L(x)$ and the Markov constant $M(f)$ are
invariant under the actions of $\GLTZ$, the approximation number
$C(x)$ is only invariant under the group of \emph{$\Z$-affine
  transformations}
\begin{equation}
  \label{eq:Z-affine}
  x\;\longmapsto\; \pm\,x\;+\;n\qquad (n\;\in\;\Z).
\end{equation}
In fact, the irrational numbers with $C(x)>\frac{1}{3}$ classified by
Gurwood are special representatives of the $\GLTZ$-classes of Markov
irrationals. In this article, we focus on the worst approximable
rational numbers.

The connection between continued fractions and Diophantine
approximation on the one hand, and hyperbolic geometry on the other
hand is well established~\cite{bonahon09,einsiedler11,hatcher}. Both
Markov forms and Markov irrationals correspond to geodesics in the
modular torus~\cite{Aigner13, cohn55, cohn71, gorshkov77, haas86,
  series85}: Markov forms correspond to simple closed geodesics, and
Markov irrationals correspond to simple geodesics with one end in the
cusp and the other end spiraling into a simple closed geodesic. The
Lagrange number $L(x)$ and the Markov constant $M(f)$ correspond to
distances between geodesics and the Ford circles interpreted as
horocycles. This makes it possible to classify the Markov irrationals
and Markov forms using purely geometric
methods~\cite{Springborn_Markov}. In this article, we take the same
geometric approach to classify the worst approximable rational
numbers.

Figure~\ref{fig:ford}
\begin{figure}[p]
  \centering
  \input{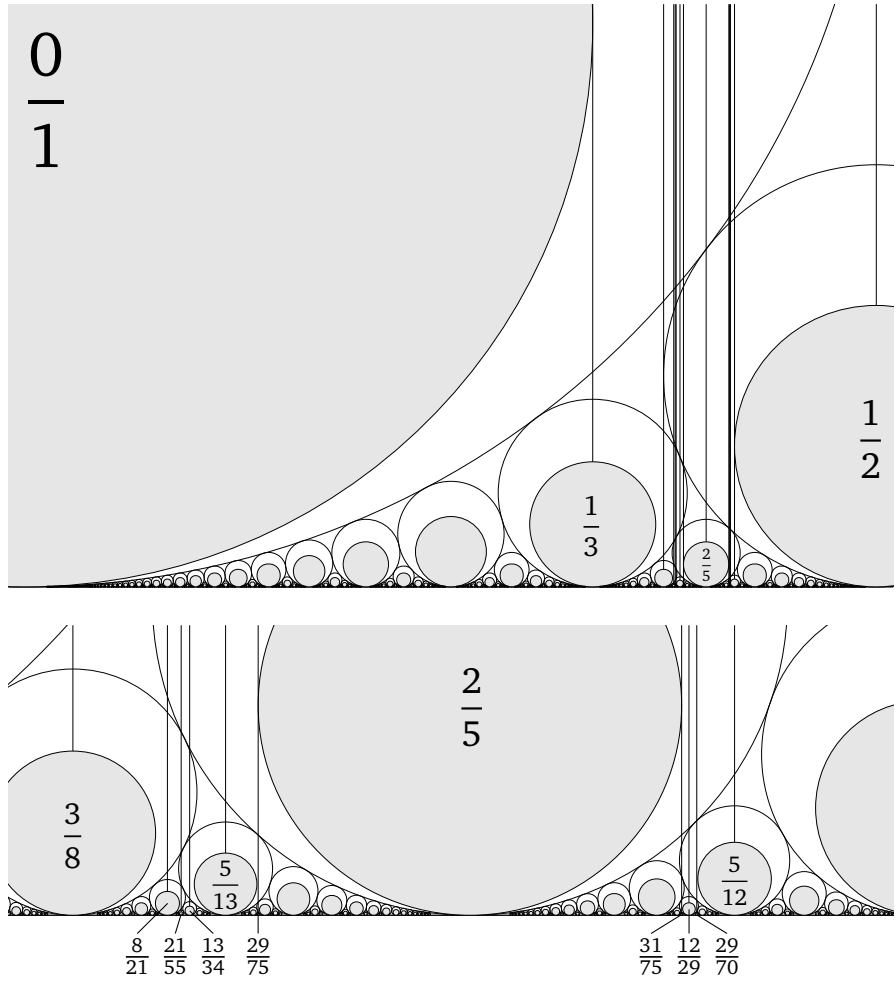}
  \caption{Ford circles scaled by $\frac{2}{3}$. Those whose highest
    point is visible from $\sminfty$ belong to rational numbers with
    approximation constant $\geq\frac{1}{3}$. A few of them are
    indicated in the figure (compare
    Figure~\ref{fig:Markov_fraction_tree}
    and Table~\ref{tab:companions}). The bottom figure shows a detail of the
    top figure.}
  \label{fig:ford}
\end{figure}
\begin{figure}[p]
  \small
  \begin{minipage}[b]{0.65\textwidth}
    \hspace*{\fill}%
    \renewcommand{\arraystretch}{1.5}%
    \begin{tabular}{c}
      \multirow{2}{*}{\quad$\dfrac{1}{1}$\quad} \\ \\
      \hline
      \multirow{2}{*}{\quad$\dfrac{0}{1}$\quad} \\ \\
    \end{tabular}
    \hspace{\fill}%
    \begin{tabular}{cc}
      \multirow{2}{*}{\quad$\dfrac{p_{2}}{q_{2}}$\quad} & \\\cline{2-2}
                                                        & \multicolumn{1}{|c}{\multirow{2}{*}{\quad$\dfrac{p'}{q'}$\quad}}\\\cline{1-1}
      \multirow{2}{*}{\quad$\dfrac{p_{1}}{q_{1}}$\quad} & \multicolumn{1}{|c}{}\\\cline{2-2}
                                                        & 
    \end{tabular}
    \hspace*{\fill}%
  \end{minipage}%
  \hspace{\fill}%
  \begin{minipage}[b]{0.35\textwidth}
    \begin{equation}
      \label{eq:pqprime}
      \begin{split}
        p' \;&=\;
        \frac{p_{1}\,q_{1}\,+\,p_{2}\,q_{2}}{p_{2}\,q_{1}\,-\,p_{1}\,q_{2}}\\[\medskipamount]
        q' \;&=\;
        \frac{q_{1}^{2}\,+\,q_{2}^{2}}{p_{2}\,q_{1}\,-\,p_{1}\,q_{2}}
      \end{split}
    \end{equation}
  \end{minipage}
  
  \caption{Root edge and generating rule for the tree of Markov
    fractions in~$[0,1]$, of which one half is shown in
    Figure~\ref{fig:Markov_fraction_tree}}
  \label{fig:Markov_fraction_tree_0_1}
\end{figure}
\begin{figure}[p]
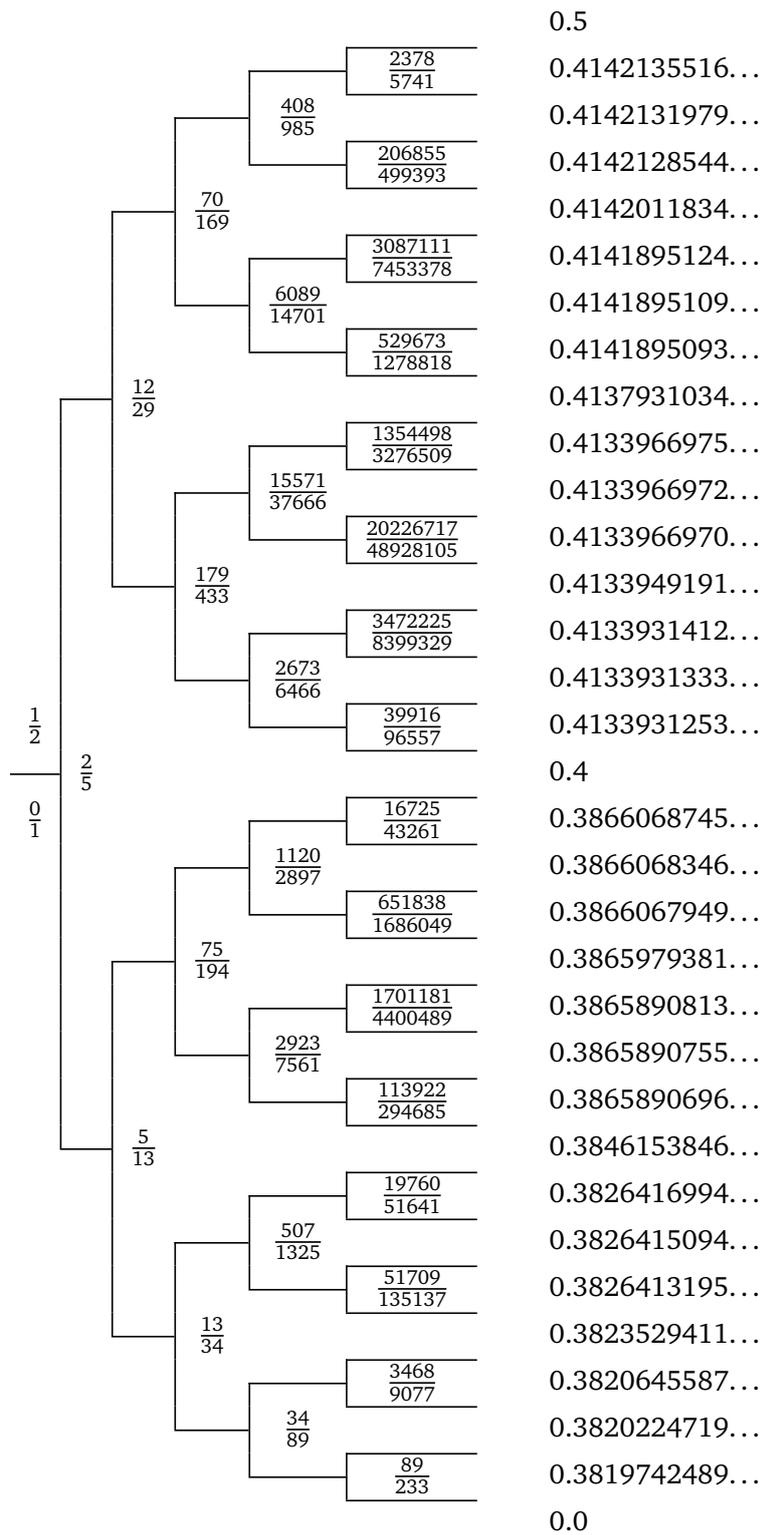

  \vspace{-\baselineskip}
  \centering
  \begingroup
  \renewcommand{\arraystretch}{.65}
%
%
\begingroup
\newcommand{\I}[1]{\multicolumn{1}{|c}{#1}}
\newcommand{\mr}[1]{\multirow{2}{*}{#1}}
\begin{tabular}{ccccccl}
               &                    &                    &                     &                     &                     & \mr{{\qquad}0.5}\\
               &                    &                    &                     &                     &                     & \\\cline{6-6}
               &                    &                    &                     &                     & \I{\mr{$\frac{2378}{5741}$}} & \mr{{\qquad}0.4142135516\ldots}\\\cline{5-5}
               &                    &                    &                     & \I{}               & \I{}               & \\\cline{6-6}
               &                    &                    &                     & \I{\mr{$\frac{408}{985}$}} &                     & \mr{{\qquad}0.4142131979\ldots}\\\cline{4-4}
               &                    &                    & \I{}               & \I{}               &                     & \\\cline{6-6}
               &                    &                    & \I{}               & \I{}               & \I{\mr{$\frac{206855}{499393}$}} & \mr{{\qquad}0.4142128544\ldots}\\\cline{5-5}
               &                    &                    & \I{}               &                     & \I{}               & \\\cline{6-6}
               &                    &                    & \I{\mr{$\frac{70}{169}$}} &                     &                     & \mr{{\qquad}0.4142011834\ldots}\\\cline{3-3}
               &                    & \I{}              & \I{}               &                     &                     & \\\cline{6-6}
               &                    & \I{}              & \I{}               &                     & \I{\mr{$\frac{3087111}{7453378}$}} & \mr{{\qquad}0.4141895124\ldots}\\\cline{5-5}
               &                    & \I{}              & \I{}               & \I{}               & \I{}               & \\\cline{6-6}
               &                    & \I{}              & \I{}               & \I{\mr{$\frac{6089}{14701}$}} &                     & \mr{{\qquad}0.4141895109\ldots}\\\cline{4-4}
               &                    & \I{}              &                     & \I{}               &                     & \\\cline{6-6}
               &                    & \I{}              &                     & \I{}               & \I{\mr{$\frac{529673}{1278818}$}} & \mr{{\qquad}0.4141895093\ldots}\\\cline{5-5}
               &                    & \I{}              &                     &                     & \I{}               & \\\cline{6-6}
               &                    & \I{\mr{$\frac{12}{29}$}}&                     &                     &                     & \mr{{\qquad}0.4137931034\ldots}\\\cline{2-2}
               & \I{}              & \I{}              &                     &                     &                     & \\\cline{6-6}
               & \I{}              & \I{}              &                     &                     & \I{\mr{$\frac{1354498}{3276509}$}} & \mr{{\qquad}0.4133966975\ldots}\\\cline{5-5}
               & \I{}              & \I{}              &                     & \I{}               & \I{}               & \\\cline{6-6}
               & \I{}              & \I{}              &                     & \I{\mr{$\frac{15571}{37666}$}} &                     & \mr{{\qquad}0.4133966972\ldots}\\\cline{4-4}
               & \I{}              & \I{}              & \I{}               & \I{}               &                     & \\\cline{6-6}
               & \I{}              & \I{}              & \I{}               & \I{}               & \I{\mr{$\frac{20226717}{48928105}$}} & \mr{{\qquad}0.4133966970\ldots}\\\cline{5-5}
               & \I{}              & \I{}              & \I{}               &                     & \I{}               & \\\cline{6-6}
               & \I{}              & \I{}              & \I{\mr{$\frac{179}{433}$}} &                     &                     & \mr{{\qquad}0.4133949191\ldots}\\\cline{3-3}
               & \I{}              &                    & \I{}               &                     &                     & \\\cline{6-6}
               & \I{}              &                    & \I{}               &                     & \I{\mr{$\frac{3472225}{8399329}$}} & \mr{{\qquad}0.4133931412\ldots}\\\cline{5-5}
               & \I{}              &                    & \I{}               & \I{}               & \I{}               & \\\cline{6-6}
               & \I{}              &                    & \I{}               & \I{\mr{$\frac{2673}{6466}$}} &                     & \mr{{\qquad}0.4133931333\ldots}\\\cline{4-4}
               & \I{}              &                    &                     & \I{}               &                     & \\\cline{6-6}
\mr{$\frac{1}{2}$} & \I{}              &                    &                     & \I{}               & \I{\mr{$\frac{39916}{96557}$}} & \mr{{\qquad}0.4133931253\ldots}\\\cline{5-5}
               & \I{}              &                    &                     &                     & \I{}               & \\\cline{6-6}
               & \I{\mr{$\frac{2}{5}$}} &                    &                     &                     &                     & \mr{{\qquad}0.4}\\\cline{1-1}
               & \I{}              &                    &                     &                     &                     & \\\cline{6-6}
\mr{$\frac{0}{1}$}  & \I{}              &                    &                     &                     & \I{\mr{$\frac{16725}{43261}$}} & \mr{{\qquad}0.3866068745\ldots}\\\cline{5-5}
               & \I{}              &                    &                     & \I{}               & \I{}               & \\\cline{6-6}
               & \I{}              &                    &                     & \I{\mr{$\frac{1120}{2897}$}} &                     & \mr{{\qquad}0.3866068346\ldots}\\\cline{4-4}
               & \I{}              &                    & \I{}               & \I{}               &                     & \\\cline{6-6}
               & \I{}              &                    & \I{}               & \I{}               & \I{\mr{$\frac{651838}{1686049}$}} & \mr{{\qquad}0.3866067949\ldots}\\\cline{5-5}
               & \I{}              &                    & \I{}               &                     & \I{}               & \\\cline{6-6}
               & \I{}              &                    & \I{\mr{$\frac{75}{194}$}} &                     &                     & \mr{{\qquad}0.3865979381\ldots}\\\cline{3-3}
               & \I{}              & \I{}              & \I{}               &                     &                     & \\\cline{6-6}
               & \I{}              & \I{}              & \I{}               &                     & \I{\mr{$\frac{1701181}{4400489}$}} & \mr{{\qquad}0.3865890813\ldots}\\\cline{5-5}
               & \I{}              & \I{}              & \I{}               & \I{}               & \I{}               & \\\cline{6-6}
               & \I{}              & \I{}              & \I{}               & \I{\mr{$\frac{2923}{7561}$}} &                     & \mr{{\qquad}0.3865890755\ldots}\\\cline{4-4}
               & \I{}              & \I{}              &                     & \I{}               &                     & \\\cline{6-6}
               & \I{}              & \I{}              &                     & \I{}               & \I{\mr{$\frac{113922}{294685}$}} & \mr{{\qquad}0.3865890696\ldots}\\\cline{5-5}
               & \I{}              & \I{}              &                     &                     & \I{}               & \\\cline{6-6}
               & \I{}              & \I{\mr{$\frac{5}{13}$}} &                     &                     &                     & \mr{{\qquad}0.3846153846\ldots}\\\cline{2-2}
               &                    & \I{}              &                     &                     &                     & \\\cline{6-6}
               &                    & \I{}              &                     &                     & \I{\mr{$\frac{19760}{51641}$}} & \mr{{\qquad}0.3826416994\ldots}\\\cline{5-5}
               &                    & \I{}              &                     & \I{}               & \I{}               & \\\cline{6-6}
               &                    & \I{}              &                     & \I{\mr{$\frac{507}{1325}$}}  &                     & \mr{{\qquad}0.3826415094\ldots}\\\cline{4-4}
               &                    & \I{}              & \I{}               & \I{}               &                     & \\\cline{6-6}
               &                    & \I{}              & \I{}               & \I{}               & \I{\mr{$\frac{51709}{135137}$}}  & \mr{{\qquad}0.3826413195\ldots}\\\cline{5-5}
               &                    & \I{}              & \I{}               &                     & \I{}               & \\\cline{6-6}
               &                    & \I{}              & \I{\mr{$\frac{13}{34}$}}  &                     &                     & \mr{{\qquad}0.3823529411\ldots}\\\cline{3-3}
               &                    &                    & \I{}               &                     &                     & \\\cline{6-6}
               &                    &                    & \I{}               &                     & \I{\mr{$\frac{3468}{9077}$}}  & \mr{{\qquad}0.3820645587\ldots}\\\cline{5-5}
               &                    &                    & \I{}               & \I{}               & \I{}               & \\\cline{6-6}
               &                    &                    & \I{}               & \I{\mr{$\frac{34}{89}$}}  &                     & \mr{{\qquad}0.3820224719\ldots}\\\cline{4-4}
               &                    &                    &                     & \I{}               &                     & \\\cline{6-6}
               &                    &                    &                     & \I{}               & \I{\mr{$\frac{89}{233}$}}  & \mr{{\qquad}0.3819742489\ldots}\\\cline{5-5}
               &                    &                    &                     &                     & \I{}               & \\\cline{6-6}
               &                    &                    &                     &                     &                     & \mr{{\qquad}0.0}\\
               &                    &                    &                     &                     &                     & \\
\end{tabular}
\endgroup
%
%

  \endgroup
  \caption{Markov fractions in the interval
    $[0,\frac{1}{2}]$. Numerical values are shown in the right
    column.
  }
  \label{fig:Markov_fraction_tree}
\end{figure}
illustrates the geometric interpretation of the approximation
constant. The extended real line $\RP^{1}=\R\cup\{\infty\}$ is the
ideal boundary of the hyperbolic plane $H^{2}$ in the half-plane
model. For an irrational number $x\in\R\setminus\Q$ and a bound $c>0$,
the approximation constant satisfies the inequality $C(x)\geq c$ if
and only if $x$ is visible from $\sminfty$ behind the Ford circles
after they are scaled by $2c$. For a rational number, $x\in\Q$, the
condition is that $x$ is visible after the Ford circle at~$x$ itself
is removed, or equivalently, that the highest point of the Ford circle
at~$x$ is visible

The classical approach to Diophantine approximation based on continued
fractions can also be interpreted geometrically, because continued
fractions encode the symbolic dynamics of geodesics in the Farey
triangulation. But our geometric approach goes further. In effect, we
trivialize the symbolic dynamics by considering not only one but all
ideal triangulations of the modular torus.

\section{Overview}
\label{sec:overview}

\subsection{Classification of the worst approximable rational numbers}
\label{sec:overview_classify}

The main result of this article is the classification of the worst
approximable rational numbers in terms of \emph{Markov fractions} and
their \emph{companions}. These special numbers can be described in an
elementary way as follows.

\paragraph{The forest of Markov fractions.}

The Markov fractions between zero and one are naturally associated
with the faces of a plane binary tree. Its root edge and generating
rule are shown in Figure~\ref{fig:Markov_fraction_tree_0_1}. For
example, the first branching after the root edge produces the Markov
fraction~$\frac{1}{2}$. At each node, the numbers $p'$ and $q'$
obtained by equations~\eqref{eq:pqprime} are coprime integers, so the
Markov fraction $\frac{p'}{q'}$ is reduced.

Figure~\ref{fig:Markov_fraction_tree} shows the first few levels of
the subtree containing the Markov fractions in the
interval~$[0,\frac{1}{2}]$. There is a symmetric subtree containing
the Markov fractions in~$[\frac{1}{2},1]$, obtained by a reflection
and replacing each Markov fraction $x$ with $1-x$. Likewise, the
Markov fractions $x+n$ in any other integer interval $[n,n+1]$ are
arranged in a similar tree, and all these trees form a plane binary
forest. The generating rule is the same for all trees, and the root
edges are infinite parallel rays separating faces associated with
consecutive integers (see Figure~\ref{fig:roots}).
\begin{figure}
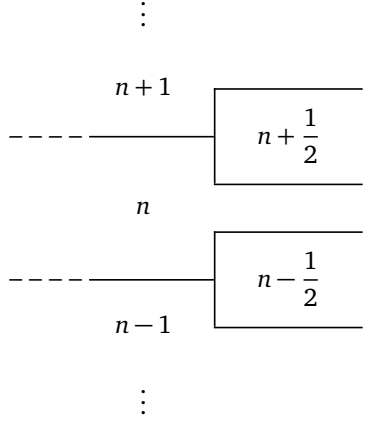

  \begin{center}
    \small%
    \renewcommand{\arraystretch}{1.5}%
    \begin{tabular}{r@{}cc}
      & $\vdots$ & \\
      & \multirow{2}{*}{\quad$n+1$\quad} & \\
      \cline{3-3}
      \multirow{2}{*}{\raisebox{-.2pt}{--\;--\;--\;--\;}} & & \multicolumn{1}{|c}{\multirow{2}{*}{\quad$n+\dfrac{1}{2}$\quad}}\\
      \cline{2-2}
      &\multirow{3}{*}{\quad$n$\quad}
               & \multicolumn{1}{|c}{}\\\cline{3-3}
      &         & \\\cline{3-3}
      \multirow{2}{*}{\raisebox{-.2pt}{--\;--\;--\;--\;}} & & \multicolumn{1}{|c}{\multirow{2}{*}{\quad$n-\dfrac{1}{2}$\quad}}\\
      \cline{2-2}
      &\multirow{2}{*}{\quad$n-1$\quad}
               & \multicolumn{1}{|c}{}\\\cline{3-3}
      &         &  \\
      &$\vdots$ &
    \end{tabular}
  \end{center}
  \caption{Roots of the forest of Markov fractions}
  \label{fig:roots}
\end{figure}

Defining Markov fractions by this plane forest is simple but not very
illuminating, and it would be a bad starting point from which to develop
the theory. So in Section~\ref{sec:Markov_fractions}, we define the
Markov fractions in terms of \emph{rational Markov triples}
(Definition~\ref{sec:Markov_fractions}) and then construct the forest
of Markov fractions in Lemma~\ref{lem:children_parents}.

At each node in the plane forest, the three Markov fractions form a
rational Markov triple
$\big(\frac{p_{1}}{q_{1}},\frac{p'}{q'},\frac{p_{2}}{q_{2}}\big)$. The
three denominators $q_{1}$, $q'$, $q_{2}$ are Markov numbers forming a
Markov triple. The rational Markov triples at the vertices are
also~\emph{centered}, i.e., the middle denominator $q'$ is largest
(see Definition~\ref{def:Markov_triple_types}).  All centered rational
Markov triples arise in this way, except for the integer triples
$\big(n-1,n,n+1\big)\in\Z^{3}$ which correspond to pairs of
neighboring root edges. The best approximating rational numbers for
the middle Markov fraction $\frac{p'}{q'}$ are the two Markov
fractions $\frac{p_{1}}{q_{1}}$ and $\frac{p_{2}}{q_{2}}$ (see
Theorem~\ref{thm:Markov_best_approximants}).

\paragraph{The sequences of companions.} For each Markov fraction
$\frac{p}{q}$, there are two infinite sequences of companions,
$(\gamma^{+}_{k}(\frac{p}{q}))_{k\geq 2}$ and
$(\gamma^{-}_{k}(\frac{p}{q}))_{k\geq 2}$ defined as follows.

\begin{definition}
  \label{def:companion}
  The \emph{right companions} and \emph{left companions} of a Markov
  fraction~$\frac{p}{q}$ are the rational numbers
  $\gamma^{+}_{k}(\frac{p}{q})$ and $\gamma^{-}_{k}(\frac{p}{q})$
  for $k\geq 2$, defined by
  \begin{equation}
    \label{eq:companions}
    \gamma^{\pm}_{k}\,\Big(\frac{p}{q}\Big)\;=\;
    \frac{p}{q}\;\pm\;
    \frac{u_{k-1}}{q\,u_{k}},
  \end{equation}
  where $(u_{k})$ is the integer sequence defined by the second
  order linear recursion
  \begin{equation}
    \label{eq:un}
    u_{0}\;=\;0,\qquad u_{1}\;=\;1,\qquad
    u_{k+1}\;=\;3q\,u_{k}\;-\;u_{k-1}\,.
  \end{equation}
\end{definition}

Note that $\gamma^{\pm}_{1}\,\Big(\frac{p}{q}\Big)=\frac{p}{q}$ is the
Markov fraction itself, so the sequences of companions start with
$k=2$ (see also Remark~\ref{rem:companions_index}).

The first companions $\gamma^{\pm}_{2}(\frac{p}{q})$ are best
approximated by the associated Markov fraction~$\frac{p}{q}$ and no
other rational number. All other companions,
$\gamma^{\pm}_{k}(\frac{p}{q})$ with $k>2$, are best approximated by
two rational numbers: the associated Markov fraction~$\frac{p}{q}$ and
the previous companion in the sequence,
$\gamma^{\pm}_{k-1}(\frac{p}{q})$ (see
Theorem~\ref{thm:companion_approximation}).

\bigskip\noindent%
We can now state the main theorem:

\begin{theorem}
  \label{thm:classify}
  For a rational number $\frac{p}{q}\in\Q$, the following two statements are
  equivalent:

  \smallskip
  \begin{compactenum}[(i)]
  \item $C(\frac{p}{q})\geq\frac{1}{3}$
  \item The number $\frac{p}{q}$ is either a Markov fraction or a
    companion of a Markov fraction.
  \end{compactenum}\smallskip
  Furthermore, $C(\frac{p}{q})=\frac{1}{3}$ if and only if $x$ is the
  first left or right companion $\gamma^{\pm}_{2}(\frac{p}{q})$ of a
  Markov fraction $x=\frac{p}{q}$.
\end{theorem}

The proof of Theorem~\ref{thm:classify} is presented in
Section~\ref{sec:pf_classify}. It relies on the characterization of
Markov fractions and their companions in terms of geodesics in the
modular torus (see Section~\ref{sec:geometric_pov}).

Because a number-theoretic question deserves a number-theoretic answer,
we first define the Markov fractions and their companions
algebraically in Section~\ref{sec:algebraic} and then derive the
geometric characterization in Section~\ref{sec:geometric_pov}. A more
geometrically minded reader might want to take the geometric
characterization as the definition and read
Sections~\ref{sec:algebraic} and~\ref{sec:geometric_pov} in reverse
order.

Yet another characterization is presented in
Section~\ref{sec:triangle_paths}, where Markov numbers and their
companions are described by a Fibonacci-like recursion along paths in
a triangle lattice. This represents the combinatorial point of view on
Markov theory and illuminates the link between Markov numbers and
Christoffel words~\cite{Reutenauer19}.

Interest in classifying the worst approximable rational numbers was
initially motivated by the \emph{uniqueness conjecture for Markov
  numbers}~\cite{Aigner13,McShane_Parlier08,Zagier82}. This is the
proposition --- yet to be proved or refuted --- that every Markov
number is the maximum of exactly one ordered Markov triple. Indeed,
the uniqueness conjecture for Markov numbers is equivalent to the
following uniqueness conjecture for Markov fractions:

\begin{conjecture}
  \label{con:markov_fractions}
  For each Markov number~$q$, the interval $[0,\frac{1}{2}]$ contains
  at most one Markov fraction with denominator $q$ .
\end{conjecture}

One may equivalently write ``exactly one'' instead of ``at most one'',
because the Markov fractions come in $\Z$-affine orbits and every
Markov number appears among the denominators (see
Section~\ref{sec:Markov_fractions}).

\subsection{Limits of infinite paths in the forest}
\label{sec:path_limit}

The forest structure of the set of Markov fractions is useful to study
the set of its accumulation points. But first note that every
accumulation point of the set of Markov fractions is
irrational. Indeed, for $\frac{a}{b}\in\Q$ and any Markov
fraction~$\frac{p}{q}\not=\frac{a}{b}$, we have
$\big|\frac{p}{q}-\frac{a}{b}\big|>\frac{1}{3b^{2}}$ because
$C(\frac{p}{q})>\frac{1}{3}$. So $\frac{a}{b}$ is not an accumulation
point of the set of Markov fractions.

Now consider infinite paths in the tree of Markov fractions between
$0$ and~$1$, starting at the root edge. Each path corresponds to a
sequence of Markov fractions with increasing denominators, which
converges to an irrational number~$x$ with Lagrange number
$L(x)\geq\frac{1}{3}$.

For example, always taking the lower branch in
Figure~\ref{fig:Markov_fraction_tree} generates the sequence
\begin{equation*}
  \frac{1}{2}\,,\; \frac{2}{5}\,,\; \frac{5}{13}\,,\; \frac{13}{34}\,,\;
  \ldots
\end{equation*}
which converges to
\begin{equation*}
  \phi^{-2}\;=\;\frac{1}{2}\,\big(3-\sqrt{5}\big)\;=\;2\,-\,\phi\,,
\end{equation*}
with $\phi=\frac{1}{2}(1+\sqrt{5})$, the golden ratio. This is a Markov
irrational with Lagrange number
$L(2-\phi)=L(\phi)=\frac{1}{\sqrt{5}}>\frac{1}{3}$.

Always taking the upper branch in
Figure~\ref{fig:Markov_fraction_tree} generates the sequence
\begin{equation*}
  \frac{2}{5}\,,\;\frac{12}{29}\,,\;\frac{70}{169}\,,\;\frac{408}{985}\,,\;\ldots
\end{equation*}
with limit $-1+\sqrt{2}$. This is another Markov irrational, with
Lagrange number 
$L(-1+\sqrt{2})=L(\sqrt{2})=\frac{1}{\sqrt{8}}>\frac{1}{3}$.

For these two examples, the limits represent the first two worst
approximable classes of irrational numbers. But there are only a
countable number of $\PGLTZ$-classes of Markov irrationals and an
uncountable number of paths in an infinite binary tree, so these two
examples are exceptional. In fact, the limit $x$ is a Markov
irrational if the path turns always left or always right except for
finitely many branchings. If, on the other hand, there are infinitely
many left turns and infinitely many right turns, then the limit has
Lagrange number $L(x)=\frac{1}{3}$. Before we come back to this in
Section~\ref{sec:inf_paths_cont}, let us use the Stern--Brocot tree to
label Markov fractions.

\subsection{Labeling the Markov fractions in the unit interval}
\label{sec:labeling}

The denominators of Markov fractions are Markov numbers. If we take
the subtree of Markov fractions in the interval $[0,\frac{1}{2}]$
shown in Figure~\ref{fig:Markov_fraction_tree} and replace the
fractions by their denominators alone, we obtain one of the six
symmetric subtrees of the tree of Markov numbers shown in
Figure~\ref{fig:Markov_tree}.
\begin{figure}
  \hspace*{\fill}\begin{tikzpicture}[grow cyclic]
  \tikzstyle{level 1}=[level distance=25pt,sibling angle=120] 
  \tikzstyle{level 2}=[level distance=25pt,sibling angle=80] 
  \tikzstyle{level 3}=[level distance=20pt,sibling angle=60] 
  \tikzstyle{level 4}=[level distance=15pt,sibling angle=60] 
  \tikzstyle{level 5}=[level distance=15pt,sibling angle=60] 
  \tikzstyle{every node}=[font=\small]
  \coordinate
  child foreach \x in {1,2,3}
  {child foreach \x in {1,2}
    {child foreach \x in {1,2}
      {child foreach \x in {1,2}
        {child foreach \x in {1,2}}}}};
  
  \draw (-35pt,0pt) node {1};
  \draw[rotate=120] (-35pt,0pt) node {1};
  \draw[rotate=-120] (-35pt,0pt) node {1};
  
  \draw (58pt,0pt) node {2};
  \draw[rotate=120] (58pt,0pt) node {2};
  \draw[rotate=-120] (58pt,0pt) node {2};
  
  \draw (62pt,31pt) node {5};
  \draw (62pt,-31pt) node {5};
  \draw[rotate=120] (62pt,31pt) node {5};
  \draw[rotate=120] (62pt,-31pt) node {5};
  \draw[rotate=-120] (62pt,31pt) node {5};
  \draw[rotate=-120] (62pt,-31pt) node {5};
  
  \draw (81pt,23pt) node {29};
  \draw (81pt,-23pt) node {29};
  \draw[rotate=120] (81pt,23pt) node {29};
  \draw[rotate=120] (81pt,-23pt) node {29};
  \draw[rotate=-120] (81pt,23pt) node {29};
  \draw[rotate=-120] (81pt,-23pt) node {29};
  
  \draw (56.5pt,51pt) node {13};
  \draw (56.5pt,-51pt) node {13};
  \draw[rotate=120] (56.5pt,51pt) node {13};
  \draw[rotate=120] (56.5pt,-51pt) node {13};
  \draw[rotate=-120] (56.5pt,51pt) node {13};
  \draw[rotate=-120] (56.5pt,-51pt) node {13};
  
  \draw (95pt,10pt) node {169};
  \draw (95pt,-10pt) node {169};
  \draw[xshift=0.5pt,yshift=0pt] [rotate=120] (95pt,10pt) node {169};
  \draw[xshift=-2pt,yshift=-3pt] [rotate=120] (95pt,-10pt) node {169};
  \draw[xshift=0pt,yshift=4pt] [rotate=-120] (95pt,10pt) node {169};
  \draw[xshift=1pt,yshift=0pt] [rotate=-120] (95pt,-10pt) node {169};
  
  \draw (89pt,38pt) node {433};
  \draw (89pt,-38pt) node {433};
  \draw[xshift=-1pt,yshift=-3pt] [rotate=120] (89pt,38pt) node {433};
  \draw[xshift=2.5pt,yshift=0pt] [rotate=120] (89pt,-38pt) node {433};
  \draw[xshift=2pt,yshift=0pt] [rotate=-120] (89pt,38pt) node {433};
  \draw[xshift=0pt,yshift=3.5pt] [rotate=-120] (89pt,-38pt) node {433};
  
  \draw (76pt,53pt) node {194};
  \draw (76pt,-53pt) node {194};
  \draw[xshift=-2pt,yshift=-3pt] [rotate=120]  (76pt,53pt) node {194};
  \draw[xshift=2pt,yshift=1pt] [rotate=120]  (76pt,-53pt) node {194};
  \draw[xshift=2.5pt,yshift=-1pt] [rotate=-120]  (76pt,53pt) node {194};
  \draw[xshift=-1pt,yshift=4pt] [rotate=-120]  (76pt,-53pt) node {194};
  
  \draw (46pt,63pt) node {34};
  \draw (46pt,-63pt) node {34};
  \draw (32pt,71pt) node {34};
  \draw (32pt,-71pt) node {34};
  \draw (-79pt,8pt) node {34};
  \draw (-79pt,-8pt) node {34};
  
  
  
  
\end{tikzpicture}  
\hspace*{\fill}\\[-4\baselineskip]
  \hspace*{\fill}%
  \makebox[0pt][r]{\small\begin{tikzpicture}
  [
  line width=0.5pt,
  every node/.style={inner sep=0pt, minimum size=0pt,
    line width=0pt}
  ]
  \draw (0,0) node (l) {}
  node [left=0.2] {$q_{1}$}
  --
  node [above=.2] {$q_{2}$}
  node [below=.2] {$q_{3}$}
  +(.5,0) node (r) {}
  node [right=0.2] {$q_{1}\smash{'}$};
  \draw (l) -- +(120:0.5);
  \draw (l) -- +(-120:0.5);
  \draw (r) -- +(60:0.5);
  \draw (r) -- +(-60:0.5);
\end{tikzpicture}
}%
  \hspace*{24pt}\\[.5\baselineskip]
  \hspace*{\fill}\makebox[0pt][r]{\small ${q_{1}}'=3q_{2}\,q_{3}-q_{1}=\frac{{q_{2}}^{2}+{q_{3}}^{2}}{q_{1}}$}
  \caption{The tree of Markov numbers and its generating rule}
  \label{fig:Markov_tree}
\end{figure}
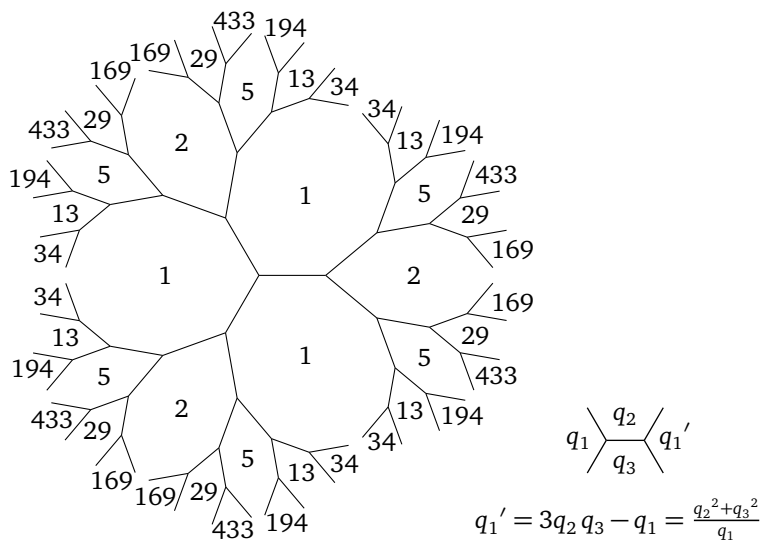
The combinatorial isomorphism with the
Stern--Brocot tree leads to the customary labeling of Markov numbers
by the rational numbers in the interval $[0,1]$. If we do not forget
about the numerators of the Markov fractions, we obtain a labeling of
the Markov fractions in the interval $[0,\frac{1}{2}]$.  For our
purposes, it makes sense to extend the Stern--Brocot tree to contain
all non-negative rational numbers and $\frac{1}{0}=\infty$. That is,
we consider the rooted trivalent plane tree with the following root
edge and generating rule:
\\
\begingroup%
\small \hspace*{\fill}%
\renewcommand{\arraystretch}{1.5}%
\begin{tabular}{c}
  \multirow{2}{*}{\quad$\dfrac{1}{0}$\quad} \\ \\
  \hline
  \multirow{2}{*}{\quad$\dfrac{0}{1}$\quad} \\ \\
\end{tabular}
\hspace{\fill}%
\begin{tabular}{cc}
  \multirow{2}{*}{\quad$\dfrac{p_{2}}{q_{2}}$\quad} & \\\cline{2-2}
                                                    & \multicolumn{1}{|c}{\multirow{2}{*}{\;$\dfrac{p_{1}\,+\,p_{2}}{q_{1}\,+\,q_{2}}$\;}}\\\cline{1-1}
  \multirow{2}{*}{\quad$\dfrac{p_{1}}{q_{1}}$\quad} & \multicolumn{1}{|c}{}\\\cline{2-2}
                                                    & 
\end{tabular}
\hspace*{\fill}%
\endgroup%
\\
The combinatorial isomorphism of rooted plane trees leads to a
bijection
\begin{equation}
  \label{eq:mu}
  \begin{split}
    \mu:\;\Q_{\geq 0}\;\cup\;\{\infty\}\quad
    &\longrightarrow\quad
    \big\{\,\text{Markov fractions}\,\big\}\;\cap\;[0,1]
    \\
    \frac{n}{m}\quad
    &\longmapsto\quad
    \mu_{\frac{n}{m}}\;.
  \end{split}
\end{equation}
The symmetries of the (extended) Stern-Brocot tree and the tree of
Markov fractions in $[0,1]$ lead to the functional equation
$\mu_{\frac{n}{m}}+\mu_{\frac{m}{n}}=1.$ For example:

\medskip
\begin{center}
  \begin{tabular}{@{\quad}c@{\quad}|@{\quad}c@{\quad}|@{\quad}c@{\quad}|@{\quad}c@{\quad}|@{\quad}c@{\quad}|@{\quad}c@{\quad}|@{\quad}c@{\quad}|@{\quad}c@{\quad}|@{\quad}c@{\quad}|@{\quad}c@{\quad}}
    $\dfrac{n}{m}$
    &$\dfrac{0}{1}$
    &$\dfrac{1}{3}$
    &$\dfrac{1}{2}$
    &$\dfrac{2}{3}$
    &$\dfrac{1}{1}$
    &$\dfrac{3}{2}$
    &$\dfrac{2}{1}$
    &$\dfrac{3}{1}$
    &$\dfrac{1}{0}$
    \\
    &&&&&&&&&\\[-2ex]
    \hline
    &&&&&&&&&\\[-2ex]
    $\displaystyle\mu_{\frac{n}{m}}$
    &$\dfrac{0}{1}$
    &$\dfrac{5}{13}$
    &$\dfrac{2}{5}$
    &$\dfrac{12}{29}$
    &$\dfrac{1}{2}$
    &$\dfrac{17}{29}$
    &$\dfrac{3}{5}$
    &$\dfrac{8}{13}$
    &$\dfrac{1}{1}$
  \end{tabular}
\end{center}

\medskip%
The same correspondence $\mu$ also arises from the construction
involving triangle paths in the Eisenstein lattice explained in
Section~\ref{sec:triangle_paths}.

\subsection{Limits of infinite paths (continued)}
\label{sec:inf_paths_cont}

The combinatorial isomorphism between the (extended) Stern--Brocot
tree and the tree of Markov fractions in the interval $[0,1]$ extends
to a bijection between infinite paths starting at the root edges of
either tree, respectively.

Each infinite path in the Stern--Brocot tree generates a sequence of
nonnegative rational numbers that converges to a finite limit or
diverges properly to~$\sminfty$. The limit of a path with infinitely
many left turns and infinitely many right turns is irrational, and
each positive irrational number is the limit for precisely one such
path. Indeed, the sequence of partial denominators of the continued
fraction expansion is the run-length encoding of the sequence of left
and right turns. On the other hand, an infinite path with only
finitely many left or finitely many right turns has a rational limit,
and every positive rational number is the limit for two such
paths. The two corresponding paths in the tree of Markov fractions
have different but $\PGLTZ$-equivalent Markov irrationals as
limits. Thus, the correspondence of infinite paths defines a map
\begin{equation*}
  \widehat{\mu}:\;\R_{\geq 0}\;\cup\;\{\infty\}
  \quad\longrightarrow\quad
  \big\{\,\text{$\PGLTZ$-classes of real numbers with $L\geq\tfrac{1}{3}$}\,\big\}.
\end{equation*}

If $x$ is rational, $\widehat{\mu}(x)$ is a class of Markov
irrationals. Otherwise, $\widehat{\mu}(x)$ is a class of irrationals
with Lagrange number $=\frac{1}{3}$. It is not difficult to see that
the restriction of $\widehat{\mu}$ to $\Q\cap[0,1]$ is a bijection
onto the set of equivalence classes of Markov irrationals. What about
the restriction of $\widehat{\mu}$ to $[0,1]\setminus\Q$? The
following two questions seem to be open:

\begin{question}
  Is every irrational number $x$ with Lagrange number
  $L(x)=\frac{1}{3}$ equivalent to the limit of some infinite path in
  the tree of Markov fractions between $0$ and $1$\,?
\end{question}

\begin{question}
  Do two different infinite paths exist in the subtree containing all
  Markov fractions between $0$ and $\frac{1}{2}$ (see
  Figure~\ref{fig:Markov_fraction_tree}) such that their limits are
  $\PGLTZ$-equivalent with Lagrange number $L=\frac{1}{3}$\,?
\end{question}

\subsection{Companions of Markov fractions}
\label{sec:companions}

Table~\ref{tab:companions}
\begin{table}[p]
  \centering
  \renewcommand{\arraystretch}{2.5}
  {\small
%
%
\begin{adjustbox}{scale=1.0,center,tabular=ll}
$\displaystyle\frac{p}{q}=\gamma^{+}_{1}\left(\frac{p}{q}\right),\quad\gamma^{+}_{2}\left(\frac{p}{q}\right),\quad\gamma^{+}_{3}\left(\frac{p}{q}\right),\quad\ldots$ & $\displaystyle\longrightarrow\quad\lim\;\gamma_{k}\left(\frac{p}{q}\right)$\\[\smallskipamount]
\hline
$\dfrac{0}{1}$,\quad$\dfrac{1}{3}$,\quad$\dfrac{3}{8}$,\quad$\dfrac{8}{21}$,\quad$\dfrac{21}{55}$,\quad$\dfrac{55}{144}$,\quad$\dfrac{144}{377}$,\quad$\dfrac{377}{987}$,\quad$\dfrac{987}{2584}$,\quad\ldots&$\longrightarrow\quad\frac{1}{2}\,\big(3-\sqrt{5}\big)$\\
$\dfrac{1}{2}$,\quad$\dfrac{7}{12}$,\quad$\dfrac{41}{70}$,\quad$\dfrac{239}{408}$,\quad$\dfrac{1393}{2378}$,\quad$\dfrac{8119}{13860}$,\quad$\dfrac{47321}{80782}$,\quad\ldots&$\longrightarrow\quad2-\sqrt{2}$\\
$\dfrac{2}{5}$,\quad$\dfrac{31}{75}$,\quad$\dfrac{463}{1120}$,\quad$\dfrac{6914}{16725}$,\quad$\dfrac{103247}{249755}$,\quad$\dfrac{1541791}{3729600}$,\quad\ldots&$\longrightarrow\quad\frac{1}{10}\,\big(19-\sqrt{221}\big)$\\
$\dfrac{5}{13}$,\quad$\dfrac{196}{507}$,\quad$\dfrac{7639}{19760}$,\quad$\dfrac{297725}{770133}$,\quad$\dfrac{11603636}{30015427}$,\quad\ldots&$\longrightarrow\quad\frac{1}{26}\,\big(49-\sqrt{1517}\big)$\\
$\dfrac{12}{29}$,\quad$\dfrac{1045}{2523}$,\quad$\dfrac{90903}{219472}$,\quad$\dfrac{7907516}{19091541}$,\quad$\dfrac{687862989}{1660744595}$,\quad\ldots&$\longrightarrow\quad\frac{1}{58}\,\big(111-\sqrt{7565}\big)$\\
$\dfrac{13}{34}$,\quad$\dfrac{1327}{3468}$,\quad$\dfrac{135341}{353702}$,\quad$\dfrac{13803455}{36074136}$,\quad$\dfrac{1407817069}{3679208170}$,\quad\ldots&$\longrightarrow\quad\frac{1}{17}\,\big(32-\sqrt{650}\big)$\\
$\dfrac{34}{89}$,\quad$\dfrac{9079}{23763}$,\quad$\dfrac{2424059}{6344632}$,\quad$\dfrac{647214674}{1693992981}$,\quad\ldots&$\longrightarrow\quad\frac{1}{178}\,\big(335-\sqrt{71285}\big)$\\
$\dfrac{70}{169}$,\quad$\dfrac{35491}{85683}$,\quad$\dfrac{17993867}{43441112}$,\quad$\dfrac{9122855078}{22024558101}$,\quad\ldots&$\longrightarrow\quad\frac{1}{338}\,\big(647-\sqrt{257045}\big)$\\
$\dfrac{75}{194}$,\quad$\dfrac{43651}{112908}$,\quad$\dfrac{25404807}{65712262}$,\quad$\dfrac{14785554023}{38244423576}$,\quad\ldots&$\longrightarrow\quad\frac{1}{97}\,\big(183-\sqrt{21170}\big)$\\
$\dfrac{89}{233}$,\quad$\dfrac{62212}{162867}$,\quad$\dfrac{43486099}{113843800}$,\quad$\dfrac{30396720989}{79576653333}$,\quad\ldots&$\longrightarrow\quad\frac{1}{466}\,\big(877-\sqrt{488597}\big)$\\
$\dfrac{179}{433}$,\quad$\dfrac{232522}{562467}$,\quad$\dfrac{302045899}{730644200}$,\quad$\dfrac{392357390279}{949106253333}$,\quad\ldots&$\longrightarrow\quad\frac{1}{866}\,\big(1657-\sqrt{1687397}\big)$\\
$\dfrac{233}{610}$,\quad$\dfrac{426391}{1116300}$,\quad$\dfrac{780295297}{2042828390}$,\quad$\dfrac{1427939967119}{3738374837400}$,\quad\ldots&$\longrightarrow\quad\frac{1}{305}\,\big(574-\sqrt{209306}\big)$\\
$\dfrac{408}{985}$,\quad$\dfrac{1205641}{2910675}$,\quad$\dfrac{3562668747}{8601043640}$,\quad$\dfrac{10527684941744}{25416081045525}$,\quad\ldots&$\longrightarrow\quad\frac{1}{1970}\,\big(3771-\sqrt{8732021}\big)$\\
\end{adjustbox}
%
%
}
  \caption{The first few right companions,
    $\gamma^{+}_{2},\,\gamma^{+}_{3},\,\ldots$, of the Markov
    fractions $\frac{p}{q}$ with denominator $q<1000$. The limits are
    Markov irrationalities, i.e., irrational numbers~$x$ with Lagrange
    number $L(x)\;>\;\tfrac{1}{3}$. The fractions in the first row are
    equal to $F_{2k-2}/F_{2k}$, and the denominators in the second row
    are equal to the Pell numbers $P_{2k}$ (see
    Remark~\ref{rem:fibonacci_pell} and compare
    Figure~\ref{fig:Markov_fraction_tree}).}
  \label{tab:companions}
\end{table}
lists the first few \emph{right companions},
$\gamma^{+}_{2}(\frac{p}{q}),\,\gamma^{+}_{3}(\frac{p}{q}),\,\ldots$,
for each Markov fraction $\frac{p}{q}=\gamma^{+}_{1}(\frac{p}{q})$
with denominator $q<1000$ (see Definition~\ref{def:companion}).  The
corresponding \emph{left companions} $\gamma^{-}_{k}(\frac{p}{q})$ are
symmetrically located:
\begin{equation}
  \label{eq:companions_symmetric}
  \gamma^{-}_{k}\,\Big(\frac{p}{q}\Big)
  \,+\,
  \gamma^{+}_{k}\,\Big(\frac{p}{q}\Big)
  \;=\;2\;\frac{p}{q}\,.
\end{equation}
The sequences of
right and left companions converge to $\frac{p}{q}\pm \delta_{q}$,
respectively, with
\begin{equation}
  \delta_{q}\;=\;\frac{3}{2}\;-\;\sqrt{\,\frac{9}{4}\;-\;\frac{1}{q^{2}}\,}\;.
\end{equation}
The closed interval
\begin{equation}
  \label{eq:Ipq}
  I_{\frac{p}{q}}\;=\;\left[\,\tfrac{p}{q}\,-\,\delta_{q},\;\tfrac{p}{q}\,+\,\delta_{q}\,\right]
\end{equation}
contains
all companions of $\frac{p}{q}$, but neither any other Markov
fractions nor their companions, but the endpoints of the interval are
limit points of the set of Markov fractions (see Lemma~\ref{lem:Ipq}).

Figure~\ref{fig:companions}
\begin{figure}
  \footnotesize
  \centering
  \input{companion_fig}
  \bigskip
  \caption{Markov fractions and their companions. Markov fractions are
    marked by vertical lines and connected to their companions by
    semicircles. The only companions that are visually discernible at
    the scales of these figures are the first six right companions of
    $\frac{0}{1}$ (see Table~\ref{tab:companions}). The vertical lines
    and semicircles project to simple geodesics in the modular torus
    with both ends in the cusp.}
  \label{fig:companions}
\end{figure}
indicates a striking self-similarity in the arrangement of Markov
fractions and their companions.

\subsection{Hyperbolic geometry}
\label{sec:overview_hyperbolic}

In the upper half-space model of the the hyperbolic plane,
\begin{equation*}
  H^{2}\;=\;\{\,z\in\C\;|\;\im\,z\; >\; 0\}\quad
  \text{with metric}\quad
  ds=\frac{|dz|}{\im z},
\end{equation*}
we associate a horocycle $h(p,q)$ with every pair $(p,q)$ of real
numbers, not both zero:
\begin{compactitem}
  \smallskip%
\item If $q\not=0$, $h(p,q)$ is the horocycle centered at
  $\frac{p}{q}$ with euclidean diameter $\frac{1}{q^{2}}$.
\item $h(p,0)$ is the horizontal horocycle centered at~$\sminfty$ with equation
  $\im z = p^{2}$.
\end{compactitem}
\smallskip%
This establishes a $\PSLTR$-equivariant
bijection between the space
\begin{equation*}
  (\R^{2}\setminus\{(0,0)\})/\{\pm 1\}
\end{equation*}
of \emph{lax vectors} (in Conway's terminology~\cite{Conway97}) and
the space of horocycles in the hyperbolic plane~\cite[p.~665]{Fock07},
\cite[Ch.~5]{Springborn_Markov}.

The following observation translates between Diophantine approximation
and hyperbolic geometry:

\begin{lemma}[{\cite[Proposition~8.1]{Springborn_Markov}}]
  \label{lem:d_vert_horo}
  The signed distance $d(g_{x},\,h(p,q))$ of the vertical
  geodesic $g_{x}$ in $H^{2}$ joining $x\in\R$ and $\sminfty$ and a
  horocycle $h(p,q)$ with $q\not=0$ and $\frac{p}{q}\not=x$ satisfies
  \begin{equation}
    \label{eq:dhg}
    d(g_{x},\,h(p,q))\;=\;\log\left(2q^{2}\left|x-\frac{p}{q}\right|\right).
  \end{equation}
\end{lemma}

\begin{proof}
  See Figure~\ref{fig:d_vert_horo}.
  \begin{figure}
    \centering
    \begin{tikzpicture}

  \coordinate[label= below:\small$x\vphantom{\frac{p}{q}}$] (x) at (0,0);
  \coordinate[label= below:$\frac{p}{q}$] (pq) at (1.5,0);

  \draw (-1,0) -- (3.5,0);
  \draw (x) -- +(0,3.5) node[pos=0.75, left] {\small$g_{x}$};
  \draw (pq) -- +(0,3.5);
  \draw ($(pq)+(0,1)$) circle (1) node[right, xshift=9.5mm, yshift=2mm] {\small$h(p,q)$};
  \draw[dashed] ($(pq)+(1.5,1.5)$) arc (0:338:1.5);

  \draw[dashed] ($(pq)+(0,2)$) -- (3.25,2) node[right]{\small$\im z=\dfrac{1}{q^{2}}$};
  \draw[dashed] ($(pq)+(0,3)$) -- (3.25,3) node[right]{\small$\im z=2\,
    \Big|\,x-\dfrac{p}{q}\,\Big|$};

  \draw (0,1.5) arc (90:0:1.5);
  \draw[very thick] (0,1.5) arc (90:67.4:1.5) node[midway, above]
  {\small$d$};
  \draw[very thick] ($(pq)+(0,2)$) -- ($(pq)+(0,3)$) node[midway,
  right] {\small$d$};
  
  \foreach \point in {x, pq}
  {
    \fill [black] (\point) circle (1pt);
  }
\end{tikzpicture}

    \caption{Signed distance $d=d(g_{x},h(p,q))$ between the vertical
      geodesic $g_{x}$ and the horocycle $h(p,q)$.}
    \label{fig:d_vert_horo}
  \end{figure}
\end{proof}

If $p$ and $q$ are coprime integers, then the horocycle $h(p,q)$ is a
\emph{Ford circle}. The previous lemma relates the approximation
number $C(x)$ to the minimal signed distance between the geodesic
$g_{x}$ and a Ford circle $h(p,q)$ not centered at either~$\sminfty$
or~$x$ (see also Figure~\ref{fig:ford}):

\begin{corollary}
  \label{cor:C_geometric}
  For $x\in\R$,
  \begin{equation}
    \label{eq:C_and_d}
    \log(2C(x))\;=\;
    \inf_{\frac{p}{q}\;\in\;\Q\,\setminus\,\{x\}}\;
    d(g_{x}, h(p,q)).
  \end{equation}
In particular, the following two statements are equivalent:
  \begin{compactenum}[(i)]
  \item $C(x)\geq\frac{1}{3}$
  \item For all Ford circles $h(p,q)$ with
    $\frac{p}{q}\in\Q\setminus\{x\}$,
    \begin{equation}
      \label{eq:dhg_ineq}
      d\,(g_{x},\,h(p,q))\;\geq\;\log\,\frac{2}{3}\,.
    \end{equation}    
  \end{compactenum}
\end{corollary}

The hyperbolic plane $H^{2}$ is the universal cover of the modular
torus $M$ (see Section~\ref{sec:modular_torus}), and the Ford circles
project to the boundary of the maximal cusp neighborhood in the
modular torus. The geometric proof of the classification
Theorem~\ref{thm:classify} for worst approximable rational numbers in
Section~\ref{sec:pf_classify} also proves the following classification
theorem of geodesics in $M$ that have both ends in the cusp but
otherwise stay furthest away from the cusp:

\begin{theorem}
  \label{thm:furthest_geodesics}
  Let $g$ be a geodesic in the modular torus $M$ with both ends in the
  cusp. Then the following statements are equivalent:
  \begin{compactenum}[(i)]
  \item Away from the ends, the signed distance between $g$ and the
    cusp is $\geq\log\frac{2}{3}$.
  \item $g$ does not intersect all simple closed geodesics in $M$.
  \item $g$ does not intersect all simple geodesics in $M$ with both
    ends in the cusp.
  \end{compactenum}
  If one and hence all of these conditions hold, then either
  \begin{compactenum}[(a)]
  \item $g$ is a simple geodesic. In this case there is a unique
    simple closed geodesic that does not intersect $g$, and infinitely
    many simple geodesics with both ends in the cusp that do not
    intersect $g$.
  \item $g$ has self-intersections, and then there are a unique closed
    geodesic $c$ in $M$ and a unique simple geodesic $\tilde{g}$ in $M$
    with both ends in the cusp that do not intersect~$g$.
  \end{compactenum}
\end{theorem}

In case (a), the geodesic $g$ in $M$ is the projection of a vertical
geodesic~$g_{x}$ in~$H^{2}$ for some Markov fraction $x$. In case (b),
the geodesic $g$ in $M$ is the projection of a vertical geodesic
$g_{x}$ in $H^{2}$ for some companion $x$ of a Markov fraction
$\tilde{x}$, and the geodesic $\tilde{g}$ in $M$ is the projection of
the vertical geodesic $g_{\tilde{x}}$ in $H^{2}$.

\section{The algebraic point of view}
\label{sec:algebraic}

This section treats the two types of worst approximable rational
numbers, the Markov fractions and their companions, from the algebraic
point of view. More geometrically minded readers might want to read
Section~\ref{sec:geometric_pov} first, which provides a complementary
geometric point of view.

\subsection{Markov numbers and Markov triples}
\label{sec:Markov_numbers}

For reference and to fix notation, we begin with a brief review of
some basic definitions and facts about Markov numbers and Markov
triples~\cite{Aigner13,Cassels57,Cusick89}.

A \emph{Markov triple} $(q_{1},q_{2},q_{3})$ is a positive integer
solution of Markov's equation
\begin{equation}
  \label{eq:Markov}
  x^{2}\,+\,y^{2}\,+\,z^{2}\;=\;3 xyz\,.
\end{equation}
Equivalently, a triple of positive integers
$(q_{1},q_{2},q_{3})\in{(\Z_{>0})}^{3}$ is a Markov triple if and only if
\begin{equation}
  \label{eq:Markov_q}
  \frac{q_{1}}{q_{2}\,q_{3}}
  \;+\;
  \frac{q_{2}}{q_{3}\,q_{1}}
  \;+\;\frac{q_{3}}{q_{1}\,q_{2}}
  \;=\; 3. 
\end{equation}
A \emph{Markov number} is a positive integer that is an element of
some Markov triple. For example, $(5,1,2)$ and $(1,5,2)$ are two of
the six Markov triples containing the Markov numbers $1$, $2$, and
$5$.

Since Markov's equation~\eqref{eq:Markov} is quadratic in each
variable, there are two solutions for each variable if the other two
variables are fixed. Thus, if $(q_{1},q_{2},q_{3})$ is a Markov
triple, then so are its \emph{neighbors}
\begin{equation*}
  (q_{1}',\,q_{2},\,q_{3}),\quad
  (q_{1},\,q_{2}',\,q_{3}),\quad
  (q_{1},\,q_{2},\,q_{3}'),
\end{equation*}
where the $q_{i}'$ are determined by Vieta's formulas for quadratic
polynomials, i.e.,
\begin{equation}
  \label{eq:qiprime}
  q_{i}'\;=\;\frac{q_{j}^{2}\,+\,q_{k}^{2}}{q_{i}}
  \;=\;3\,q_{j}\,q_{k}\,-\,q_{i}
\end{equation}
with $\{i,j,k\}=\{1,2,3\}$. This neighbor relation turns the set of
Markov triples into a plane binary tree. The Markov numbers correspond
to the regions into which the tree separates the plane (see
Figure~\ref{fig:Markov_tree}).

The root Markov triple $(1,1,1)$ and its neighbors $(2,1,1)$,
$(1,2,1)$, and $(1,1,2)$ are called \emph{singular}. All other Markov
triples are called \emph{non-singular} and consist of three different
Markov numbers. For each Markov triple, there is a unique path to the
root. At each step along this path, the largest Markov number of the
triple is replaced according to the rule~\eqref{eq:qiprime}.

\subsection{Markov fractions and rational Markov triples}
\label{sec:Markov_fractions}

Markov fractions are one of the two types of worst approximable
rational numbers (see Theorems~\ref{thm:classify}
and~\ref{thm:Markov_best_approximants}). Like the Markov numbers, we
define Markov fractions in terms of triples:

\begin{definition}
  \label{def:Markov_fraction}
  (i) A triple of rational numbers $(x_{1},x_{2},x_{3})$ is a
  \emph{rational Markov triple} if $x_{k}=\frac{p_{k}}{q_{k}}$ for
  some Markov triple $(q_{1},q_{2},q_{3})$ and some integers
  $p_{1},p_{2},p_{3}$ satisfying the equations
  \begin{align}
    \label{eq:p1p2}
    p_{2}\,q_{1}\,-\,p_{1}\,q_{2}\;&=\;q_{3}\,,\\
    \label{eq:p2p3}
    p_{3}\,q_{2}\,-\,p_{2}\,q_{3}\;&=\;q_{1}\,.
  \end{align}

  (ii) A \emph{Markov fraction} is a rational number $x_{k}$ that is an
  element of some rational Markov triple $(x_{1},x_{2},x_{3})$.
\end{definition}

A rational Markov triple $(x_{1},x_{2},x_{3})$ is
strictly increasing,
\begin{equation*}
  x_{1}\;<\;x_{2}\;<\;x_{3}\,,
\end{equation*}
because Markov numbers are positive and equations~\eqref{eq:p1p2}
and~\eqref{eq:p2p3} are equivalent to
\begin{equation}
  \label{eq:p1p2_alt}
  \frac{p_{2}}{q_{2}}\,-\,\frac{p_{1}}{q_{1}}\;=\;\frac{q_{3}}{q_{1}\,q_{2}}
  \qquad\text{and}\qquad
  \frac{p_{3}}{q_{3}}\,-\,\frac{p_{2}}{q_{2}}\;=\;\frac{q_{1}}{q_{2}\,q_{3}}.
\end{equation}
Equations~\eqref{eq:Markov_q}, \eqref{eq:qiprime},
and~\eqref{eq:p1p2_alt} also yield the relation
\begin{equation}
  \label{eq:p1p3}
  p_{3}\,q_{1}\,-\,p_{1}\,q_{3}\;=\;q_{2}'
\end{equation}
with $q_{2}'$ defined by~\eqref{eq:qiprime}, which will
be used in the proof of Lemma~\ref{lem:completion}.

Since the Markov numbers $q_{k}$ of a Markov triple
$(q_{1},q_{2},q_{3})$ are pairwise coprime, equations~\eqref{eq:p1p2}
and~\eqref{eq:p2p3} imply that each pair $p_{k}$, $q_{k}$ is also
coprime, so the fractions $\frac{p_{k}}{q_{k}}$ are reduced. This
means that the Markov triple $(q_{1}, q_{2}, q_{3})$ in
Definition~\ref{def:Markov_fraction} is uniquely determined by the
rational Markov triple $(x_{1},x_{2},x_{3})$.
Further, if $(x_{1},x_{2},x_{3})$ is a rational Markov triple, then for
every $n\in\Z$,
\begin{equation*}
  (\,x_{1}+n,\;x_{2}+n,\;x_{3}+n\,)
\end{equation*}
is also a rational Markov triple with the same Markov triple
$(q_{1},q_{2},q_{3})$ of denominators, and all rational Markov triples
with this denominator triple are obtained in this way. Since
$(-x_{3},-x_{2},-x_{1})$ is also a rational Markov triple, the group
of $\Z$-affine transformations~\eqref{eq:Z-affine} acts on the set of
rational Markov triples, and hence on the set of Markov fractions.

There is another $\Z$-action on the set of rational Markov
triples: Using equations~\eqref{eq:Markov_q} and~\eqref{eq:p1p2_alt},
it is easy to check that
\begin{equation}
  \label{eq:Zshift}
  (\,x_{3}-3,\;x_{1},\;x_{2}\,)  
  \quad\text{and}\quad
  (\,x_{2},\;x_{3},\;x_{1}+3\,)
\end{equation}
are also rational Markov triples, with cyclically permuted
denominators. So the action of the cyclic permutation group
$A_{3}\simeq\Z_{3}$ on the set of Markov triples $(q_{1},q_{2},q_{3})$
lifts to a free action of $\Z$ on the set of rational Markov triples
$\big(\frac{p_{1}}{q_{1}},\frac{p_{2}}{q_{2}},\frac{p_{3}}{q_{3}}\big)$. Accordingly,
the tree of Markov numbers (see Figure~\ref{fig:Markov_tree}) corresponds
to the forest of Markov fractions constructed in
Lemma~\ref{lem:children_parents}, a subtree of which is shown in
Figure~\ref{fig:Markov_fraction_tree}.

To see that there exists a rational Markov triple
$\big(\frac{p_{1}}{q_{1}},\frac{p_{2}}{q_{2}},\frac{p_{3}}{q_{3}}\big)$
for each Markov triple $(q_{1},q_{2},q_{3})$, note that the
coprimality of Markov triples implies the existence of $p_{1}$ and
$p_{2}$ satisfying~\eqref{eq:p1p2} (and $p_{2}$ and $p_{3}$
satisfying~\eqref{eq:p2p3}, and $p_{1}$ and $p_{3}$
satisfying~\eqref{eq:p1p3}). By the following lemma, there exists a
suitable $p_{3}$ (or $p_{1}$ or $p_{2}$, respectively) to complete
the rational Markov triple.

\begin{lemma}[Completion of rational Markov triples]
  \label{lem:completion}
  
  Let $(q_{1},q_{2},q_{3})$ be a Markov
  triple.

  \begin{compactenum}[(i)]
  \item If $p_{1},p_{2}\in\Z$ satisfy equation~\eqref{eq:p1p2}, then
    $ \frac{p_{2}^{2}\,+\,1}{q_{2}}\,\in\,\Z, $ and
    \begin{equation}\label{eq:p3}
      p_{3}\;\coloneqq\; q_{1}\cdot\,\frac{p_{2}^{2}\,+\,1}{q_{2}}\,-\,p_{1}\,p_{2}
    \end{equation}
    satisfies equation~\eqref{eq:p2p3}.
  \item If $p_{2},p_{3}\in\Z$ satisfy equation~\eqref{eq:p2p3}, then
    $ \frac{p_{2}^{2}\,+\,1}{q_{2}}\,\in\,\Z, $ and
    \begin{equation}\label{eq:p1}
      p_{1}\;\coloneqq\; p_{2}\,p_{3}\,-\,q_{3}\cdot\frac{p_{2}^{2}\,+\,1}{q_{2}}
    \end{equation}
    satisfies equation~\eqref{eq:p1p2}.
  \item If $p_{1},p_{3}\in\Z$ satisfy equation~\eqref{eq:p1p3}, then
    \begin{equation}\label{eq:p2}
      p_{2}\;\coloneqq\;\frac{p_{1}\,q_{1}\,+\,p_{3}\,q_{3}}{q_{2}'}
    \end{equation}
    is an integer and satisfies equations~\eqref{eq:p1p2} and~\eqref{eq:p2p3}.
  \end{compactenum}
  \medskip%
  In any case, therefore,
  $\big(\frac{p_{1}}{q_{1}},\frac{p_{2}}{q_{2}},\frac{p_{3}}{q_{3}}\big)$
  is a rational Markov triple, which is, moreover, uniquely determined
  by $\big(\frac{p_{1}}{q_{1}},\frac{p_{2}}{q_{2}}\big)$,
  $\big(\frac{p_{2}}{q_{2}},\frac{p_{3}}{q_{3}}\big)$, or
  $\big(\frac{p_{1}}{q_{1}},\frac{p_{3}}{q_{3}}\big)$, respectively.
\end{lemma}

For a proof see Section~\ref{sec:pf_completion_lemma}. The integrality
$\frac{p_{2}^{2}\,+\,1}{q_{2}}$ implies the following necessary condition for
Markov numbers, which is well known~\cite[Lemma~3.14]{Aigner13} and
was proved in a similar way~\cite[p.~196]{Schmutz}:

\begin{corollary}
  \label{cor:quadratic_residue}
  If $q$ is a Markov number, then $-1$ is a quadratic residue modulo
  $q$.
\end{corollary}

It will be useful to extend the distinction between singular and
non-singular Markov triples (see Section~\ref{sec:Markov_numbers}) to
rational Markov triples, and to have a special term to single out the
rational Markov triples with denominators $(1,1,1)$. Also, since we
are especially interested in rational Markov triples in which the
second denominator is largest (see
Theorem~\ref{thm:Markov_best_approximants}), it makes sense to
introduce a special term for those, too:

\begin{definition}
  \label{def:Markov_triple_types}
  A rational Markov triple
  $\big(\frac{p_{1}}{q_{1}},\frac{p_{2}}{q_{2}},\frac{p_{3}}{q_{3}}\big)$
  is called
  \begin{compactitem}
    \smallskip
  \item \emph{singular} if the Markov triple $(q_{1},q_{2},q_{3})$
    is singular, i.e., $\{q_{1},q_{2},q_{3}\}\subset\{1,2\}$,
    \smallskip
  \item \emph{integral} if $q_{1}=q_{2}=q_{3}=1$, and
    \smallskip
  \item \emph{centered} if
    \begin{equation}
      \label{eq:centered}
      q_{2}\;\geq\;\max\,\{\,q_{1},\,q_{3}\,\}.
  \end{equation}
  \end{compactitem}
\end{definition}

Integral rational Markov triples are of the form $(n-1,n,n+1)$ with
$n\in\Z$. They are integral, centered, and the only centered rational
Markov triples for which~\eqref{eq:centered} is satisfied with
equality. All non-integral singular Markov triples have one of the
forms $(n,n+\frac{1}{2},n+1)$, $(n-2,n,n+\frac{1}{2})$ or
$(n-\frac{1}{2},n,n+2)$ for $n\in\Z$, and precisely those of the
first form are centered.

\begin{remark}
  \label{rem:centered_Markov_triple}
  The notion of a centered rational Markov triple is closely related
  to Rockett~\& Sz\"usz's concept of a \emph{complete Markov
    triple}~\cite[p.~103]{Rockett92}: If
  $\big(\frac{p_{1}}{q_{1}},\frac{p_{2}}{q_{2}},\frac{p_{3}}{q_{3}}\big)$
  is a centered rational Markov triple with
  $\frac{p_{k}}{q_{k}}\in[0,\frac{1}{2}]$ then
  $(q_{2}, p_{2}; q_{1}, p_{1}; q_{3}, p_{3})$ is a complete Markov
  triple. Compare also Figure~\ref{fig:Markov_fraction_tree}
  with~\cite[p.~105]{Rockett92}. However, the idea to interpret the
  pairs $(q_{i}, p_{i})$ as rational numbers $\frac{p_{i}}{q_{i}}$
  does not seem to be obvious in the context of~\cite{Rockett92}.
\end{remark}

We can now state the main result about the approximation constant of a
Markov fraction, and the best approximating rationals that achieve it,
as follows:

\begin{theorem}[Best approximants of Markov fractions]
  \label{thm:Markov_best_approximants}
  (i) Every Markov fraction is the middle element of a
  unique centered rational Markov triple.

  (ii) If
  $\big(\frac{p_{1}}{q_{1}},\frac{p_{2}}{q_{2}},\frac{p_{3}}{q_{3}}\big)$
  is a centered rational Markov triple, then the best approximants
  of~$\frac{p_{2}}{q_{2}}$ are precisely~$\frac{p_{1}}{q_{1}}$
  and~$\frac{p_{3}}{q_{3}}$\,, i.e.,
  \begin{equation*}
    C
    \left(
      \frac{p_{2}}{q_{2}}
    \right)
    \;=\;
    q_{1}^{2}\cdot\,
    \left(
      \frac{p_{2}}{q_{2}}-\frac{p_{1}}{q_{1}}
    \right)
    \;=\;
    q_{3}^{2}\cdot\,
    \left(
      \frac{p_{3}}{q_{3}}-\frac{p_{2}}{q_{2}}
    \right)\,,
  \end{equation*}
  and for every rational number $\frac{a}{b}$ not contained in the
  triple,
  \begin{equation*}
    C
    \left(
      \frac{p_{2}}{q_{2}}
    \right)
    <
    b^{2}\cdot \left| \frac{p_{2}}{q_{2}}-\frac{a}{b} \right|.
  \end{equation*}
    In particular, this implies
  \begin{equation}
    \label{eq:Cp2q2}
    C
    \left(
      \frac{p_{2}}{q_{2}}
    \right)
    \;=\; \frac{q_{1}\,q_{3}}{q_{2}}\;>\;\frac{1}{3}\,.
  \end{equation}
\end{theorem}

\noindent%
For a proof of (i) see Section~\ref{sec:pf_thm_unique_centered_triple},
for a proof of (ii) see Section~\ref{sec:pr_Markov_approximants}.

\begin{examples}
  \label{exa:optimal}
  Since $\big(\frac{0}{1},\frac{2}{5},\frac{1}{2}\big)$ and
  $\big(\frac{2}{5},\frac{12}{29},\frac{1}{2}\big)$ are centered
  rational Markov triples,
  \begin{compactenum}[\bf (i)]
  \item the best approximants of $\frac{2}{5}$ are
    $\frac{0}{1}$ and $\frac{1}{2}$, and $C\big(\tfrac{2}{5}\big)=\tfrac{2}{5}$,
  \item the best approximants of $\frac{12}{29}$ are
    $\frac{2}{5}$ and $\frac{1}{2}$, and $C\big(\tfrac{12}{29}\big)=\tfrac{10}{29}$.
  \end{compactenum}
  Figure~\ref{fig:fundamental_domain_1-5-2} illustrates the
  geometric interpretation of these examples.
\end{examples}

The plane forest of centered rational Markov triples (see
Section~\ref{sec:overview_classify} and
Figure~\ref{fig:Markov_fraction_tree}) is constructed in the following
lemma. It is used in Section~\ref{sec:pf_thm_unique_centered_triple}
for the proof Theorem~\ref{thm:Markov_best_approximants}~(i) and in
Section~\ref{sec:markov_frac_geometric} to establish the geometric
characterization of rational Markov triples.

\begin{lemma}[Forest of centerend rational Markov triples]
  \label{lem:children_parents}
  \ 
  \begin{compactenum}[(i)]
  \item If
    $\big(\frac{p_{1}}{q_{1}},\frac{p_{2}}{q_{2}},\frac{p_{3}}{q_{3}}\big)$
    is a rational Markov triple, then so are its \emph{children}
    \begin{equation}
      \label{eq:children}
      \left(\,
        \frac{p_{1}}{q_{1}},\;
        \frac{p_{3}'}{q_{3}'},\;
        \frac{p_{2}}{q_{2}}\,
      \right)
      \quad\text{and}\quad
      \left(\,
        \frac{p_{2}}{q_{2}},\;
        \frac{p_{1}'}{q_{1}'},\;
        \frac{p_{3}}{q_{3}}\,
      \right)
    \end{equation}
    and its \emph{parents}
    \begin{equation}
      \label{eq:parents}
      \left(\,
        \frac{p_{2}'}{q_{2}'},\;
        \frac{p_{1}}{q_{1}},\;
        \frac{p_{3}}{q_{3}}\,
      \right),
      \qquad
      \left(\,
        \frac{p_{1}}{q_{1}},\;
        \frac{p_{3}}{q_{3}},\;
        \frac{p_{2}'}{q_{2}'}\,+\,3\,
      \right)  
    \end{equation}
    where $q_{1}'$, $q_{2}'$, and $q_{3}'$ are defined
    by~\eqref{eq:qiprime}, and $p_{1}'$, $p_{2}'$, and $p_{3}'$
    defined by
    \begin{align}
      \label{eq:p1prime}
      p_{1}'\;&=\;\frac{1}{q_{1}}\;(\,p_{2}\,q_{2}\,+\,p_{3}\,q_{3}\,)\,,\\
      \label{eq:p2prime}
      p_{2}'\;&=\;p_{3}\,p_{1}\,-\,\frac{p_{1}^{2}+1}{q_{1}}\;q_{3}\,,\\
      \label{eq:p3prime}
      p_{3}'\;&=\;\frac{1}{q_{3}}\;(\,p_{1}\,q_{1}\,+\,p_{2}\,q_{2}\,).
    \end{align}
    In particular, $p_{k}'\in\Z$ for $k\in\{1,2,3\}$.
    
  \item A rational Markov triple is a child of each of its parents and
    a parent of each of its children.
  
  \item Both children of a centered rational Markov triple are
    centered, and exactly one parent of a non-singular centered
    rational Markov triple is centered.
  \end{compactenum}

  \medskip%
  Thus, the parent-child relationship induces on the set of
  non-integer centered rational Markov triples the structure of a
  plane binary forest, with singular root triples $(n,
  n+\frac{1}{2}, n+1)$ as shown in Figure~\ref{fig:roots}
  and the generating rule~\eqref{eq:pqprime}.
\end{lemma}

For a proof, see Section~\ref{sec:pf_lem_forrest}.

\subsubsection{Proof of Lemma~\ref{lem:completion} (Completion of rational Markov triples)}
\label{sec:pf_completion_lemma}

To see part (i) of Lemma~\ref{lem:completion}, first note that
$q_{2}$ divides $p_{2}^{2}+1$ because $q_{2}$ and $q_{1}$ are coprime
and
\begin{equation*}
  q_{1}^{2}\;(\,p_{2}^{2}\,+\,1\,)\;
  \overset{\eqref{eq:p1p2}}{=}\;(p_{1}q_{2}+q_{3})^{2}+q_{1}^{2}
  \overset{\,\eqref{eq:Markov}\,}{=}\;
  q_{2}\;(\,p_{1}^{2}\,q_{2}\,+\,2\,p_{1}\,q_{3}\,+\,3\,q_{1}\,q_{3}\,-\,q_{2}\,)\,.
\end{equation*}
A straightforward calculation yields equation~\eqref{eq:p2p3}:
\begin{equation*}
  p_{3}\,q_{2}\,-\,p_{2}\,q_{3}\;
  \overset{\eqref{eq:p3}}{=}\;
  p_{2}\,(\,p_{2}\,q_{1}\,-\,p_{1}\,q_{2}\,)\,+\,q_{1}\,-\,p_{2}\,q_{3}\;
  \overset{\eqref{eq:p1p2}}{=}\; q_{1}
\end{equation*}
Statement (ii) can be proved in the same way.

For the integrality statement of (iii), consider the identity
\begin{equation*}
  q_{1}\,(\,p_{1}\,q_{1}\,+\,p_{3}\,q_{3}\,)\;
  \overset{\eqref{eq:p1p3}}{=}\;
  p_{1}\,(\,q_{1}^{2}\,+\,q_{3}^{2})\,+\,q_{2}'\,q_{3}\;
  \overset{\eqref{eq:qiprime}}{=}\;
  q_{2}'\,(\,p_{1}\,q_{2}\,+\,q_{3}\,)\,. 
\end{equation*}
So $q_{2}'$ divides $(p_{1}q_{1}+p_{3}q_{3})$ because $q_{1}$ and
$q_{2}'$ are coprime as elements of the Markov triple
$(q_{1},q_{2}',q_{3})$.

To verify equations~\eqref{eq:p1p2} and~\eqref{eq:p2p3}, use the
equivalent forms~\eqref{eq:p1p2_alt}: With
\begin{equation*}
  \frac{p_{2}}{q_{2}}
  \;\overset{\eqref{eq:p2}}{=}\;
  \frac{p_{1}\,q_{1}\,+\,p_{3}\,q_{3}}{q_{2}\,q_{2}'}
  \;\overset{\eqref{eq:qiprime}}{=}\;
  \frac{p_{1}\,q_{1}\,+\,p_{3}\,q_{3}}{q_{1}^{2}\,+\,q_{3}^{2}}
\end{equation*}
we obtain
\begin{equation*}
  \frac{p_{2}}{q_{2}}\,-\,\frac{p_{1}}{q_{1}}\;
  \overset{\eqref{eq:p2}}{=}\;
  \frac{(p_{3}\,q_{1}\,-\,p_{1}\,q_{3})\,q_{3}}{(\,q_{1}^{2}\,+\,q_{3}^{2}\,)\,q_{1}}\;
  \overset{\eqref{eq:p1p3}}{=}\;
  \frac{q_{2}'\,q_{3}}{(\,q_{1}^{2}\,+\,q_{3}^{2}\,)\,q_{1}}
  \overset{\eqref{eq:qiprime}}{=}\;
  \frac{q_{3}}{q_{1}\,q_{2}}\,,
\end{equation*}
proving~\eqref{eq:p1p2}. An analogous calculation
proves~\eqref{eq:p2p3}.

Finally, since it has been established that
$\big(\frac{p_{1}}{q_{1}},\frac{p_{2}}{q_{2}},\frac{p_{3}}{q_{3}}\big)$
is a rational Markov triple, the uniqueness of the completion
follows from equations~\eqref{eq:p1p2_alt}.\qed

\subsubsection{Proof of Lemma~\ref{lem:children_parents} (Forest of centered rational Markov triples)}
\label{sec:pf_lem_forrest}

Parts (i) and (ii) of Lemma~\ref{lem:children_parents} can be proved
by applying the Completion-Lemma~\ref{lem:completion}
appropriately. For example, to see that the first child is a rational
Markov triple, let
$\frac{\tilde{p}_{1}}{\tilde{q}_{1}}=\frac{p_{1}}{q_{1}}$ and
$\frac{\tilde{p}_{3}}{\tilde{q}_{3}}=\frac{p_{2}}{q_{2}}$. Then
\begin{equation*}
  \tilde{p}_{3}\,\tilde{q}_{1}\,-\,\tilde{p}_{1}\,\tilde{q}_{3}
  \;=\;q_{3}\;=:\;\tilde{q}_{2}'\,,
\end{equation*}
and Lemma~\ref{lem:completion} (iii) says that
\begin{equation*}
  \left(\,
    \frac{\tilde{p}_{1}}{\tilde{q}_{1}},\;
    \frac{\tilde{p}_{2}}{\tilde{q}_{2}},\;
    \frac{\tilde{p}_{3}}{\tilde{q}_{3}}\,
  \right)
  =
  \left(\,
    \frac{p_{1}}{q_{1}},\;
    \frac{p_{3}'}{q_{3}'},\;
    \frac{p_{2}}{q_{2}}\,
  \right)
\end{equation*}
is a rational Markov triple.

Statement (iii) follows from the same properties of Markov triples and
their neighbors that produce the tree of Markov
numbers~\cite{Cassels57}: If $q_{2}\geq\max\{q_{1},q_{3}\}$, then
$q_{1}'>q_{2}$ and $q_{3}'>q_{2}$, so both children of a centered
rational Markov triple are centered. If $q_{2}>\max\{q_{1},q_{3}\}>1$,
then $\min\{q_{1},q_{3}\}\leq q_{2}'<\max\{q_{1},q_{3}\}$, so exactly
one parent of a non-singular centered rational Markov triple is
centered.

Therefore, starting with any non-singular centered rational Markov
triple and traversing up the parental line along centered triples, one
will eventually reach a singular root triple
$(n, n+\frac{1}{2}, n+1)$.

The generating rule~\eqref{eq:pqprime} is equivalent to the equations
for children.  \qed

\subsubsection{Proof of Theorem~\ref{thm:Markov_best_approximants} (i).
  Every Markov fraction is the middle element of a unique centered
  triple}
\label{sec:pf_thm_unique_centered_triple}

By definition, every Markov fraction $\frac{p}{q}$ is contained in some
rational Markov triple. If $q$ is not the largest denominator in that
triple, iteratively replace the triple with the parent or child in
which the fraction with largest denominator is removed until~$q$ is the
largest denominator. If $\frac{p}{q}$ is not the middle element in
this triple then it is the middle element in one of the adjacent
triples~\eqref{eq:Zshift}. This shows that there exists a centered
rational Markov triple $m$ with $\frac{p}{q}$ in the middle.

To see that $m$ is the unique centered rational Markov triple with
$\frac{p}{q}$ in the middle, consider first the tree of centered
rational Markov triples that contains $m$. All descendants of $m$ have
a larger denominator in the middle, and all ancestors have a smaller
denominator in the middle. So no other centered triple in the same
tree has $\frac{p}{q}$ in the middle. But each tree in the forest of
centered rational Markov triples contains only Markov fractions in a
unique integer interval $[n,n+1]$. So if a Markov fraction is
contained in two centered rational Markov triples of different trees,
then it is an integer, and it is the last element in one triple and
the first element in the other. This shows that no two centered
rational Markov triples have the same middle element.  \qed

\subsection{Companions of Markov fractions}
\label{sec:companions_alg}

This section deals with the companions of Markov fractions (see
Definition~\ref{def:companion}) from the algebraic point of
view. Proofs in which we use the geometric interpretation are deferred
to Section~\ref{sec:geometric_pov}. 

We will often use the following lemma implicitly (e.g., in the
statement of Theorem~\ref{thm:companion_approximation}).

\begin{lemma}
  \label{lem:companion_reduced}
  The common denominator on the right hand side of
  equation~\eqref{eq:companions} is indeed $q\,u_{k}$, so the fraction representation
  \begin{equation*}
    \gamma^{\pm}_{k}\,\Big(\frac{p}{q}\Big)\;=\;
    \frac{p\,u_{k}\pm u_{k-1}}{q\,u_{k}}
  \end{equation*}
  is reduced.
\end{lemma}

\begin{proof}
  See Corollary~\ref{cor:rs_coprime}.
\end{proof}

\begin{remark}[Index issues]
  \label{rem:companions_index}
  For $k=1$, equation~\eqref{eq:companions} says
  \begin{equation}
    \label{eq:gamma_1}
    \gamma^{\pm}_{1}\,\Big(\frac{p}{q}\Big)
    \;=\;\frac{p}{q}\,,
  \end{equation}
  But Definition~\ref{def:companion} explicitly requires $k\geq 2$, so
  the Markov fraction $\frac{p}{q}$ is by definition not a companion
  of itself. Consequently, \emph{the $k$-th left and right companions}
  are $\gamma^{+}_{k+1}(\frac{p}{q})$ and
  $\gamma^{-}_{k+1}(\frac{p}{q})$, respectively. One could fix this
  notational nuisance either by changing the definition and counting
  the Markov fractions as their own companions, or by shifting the
  index in the definition of $\gamma^{\pm}_{k}$. The first option is
  bad because Markov fractions and their companions are different both
  regarding Diophantine approximation (compare
  Theorems~\ref{thm:Markov_best_approximants}
  and~\ref{thm:companion_approximation}) and from the geometric point
  of view (compare Lemmas~\ref{lem:Markov_triple_geom}
  and~\ref{lem:companions_geom}). The second option is bad because it
  would create other nuisances. For example, the
  symmetries~\eqref{eq:u_symm} and~\eqref{eq:gamma_symm} of the
  following remark would become more complicated.
\end{remark}

\begin{remark}[Symmetry]
  \label{rem:companions_symmetries}
  The recursion~\eqref{eq:un} can be used to define $u_{k}$ and
  hence $\gamma^{\pm}_{k}$ not only for $k\geq 0$ but for all
  $k\in\Z$. Then
  \begin{equation}
    \label{eq:u_symm}
    u_{-k}=-u_{k}  
  \end{equation}
  and this implies
  $\gamma^{\pm}_{0}\big(\frac{p}{q}\big)=\sminfty$ and the following
  symmetry relating $\gamma^{+}$ and $\gamma^{-}$:
  \begin{equation}
    \label{eq:gamma_symm}
    \gamma^{+}_{-k}\,\Big(\frac{p}{q}\Big)
    \;=\;
    \gamma^{-}_{k}\,\Big(\frac{p}{q}+3\Big)
  \end{equation}
\end{remark}

The following properties of the sequence $(u_{k})_{k\in\Z}$ are easily verified
by induction.

\begin{lemma}[Properties of $u_{k}$]
  \label{lem:u_prop}
  For any Markov number $q$, the sequence $(u_{k})_{k\in\Z}$ defined
  by~\eqref{eq:un} is
  \begin{compactenum}[\bf (i)]
  \item strictly increasing and
  \item satisfies the equation
    \begin{equation}
      \label{eq:u_conserved}
      u_{k+1}^{2}\,-\,3\,q\,u_{k+1}\,u_{k}\,+\,u_{k}^{2}\;=\;1
    \end{equation}
    for all $k\in\Z$.
  \end{compactenum}
\end{lemma}

\begin{remark}[Fibonacci and Pell numbers]
  \label{rem:fibonacci_pell}
  (i) For $q=1$,
  \begin{equation*}
    u_{k}\;=\;F_{2k}\quad\text{and so}\quad
    \gamma^{\pm}_{k}(p)\;=\;p\,\pm\,\frac{F_{2k-2}}{F_{2k}},
  \end{equation*}
  where $(F_{k})$ is the sequence of
  Fibonacci numbers,
  \begin{equation*}
    F_{0}\;=\;0,\quad
    F_{1}\;=\;1,\quad
    F_{k+1}\;=\;F_{k}\,+\,F_{k-1}. 
  \end{equation*}
  (ii) For $q=2$ and odd $p$,
  \begin{equation*}
    2\,u_{k}\;=\;P_{2k}
    \quad\text{and so}\quad
    \gamma^{\pm}_{k}\,\Big(\frac{p}{2}\Big)
    \;=\;
    \frac{p}{2}
    \,\pm\,
    \frac{P_{2k-2}}{2\,P_{2k}}\,,
  \end{equation*}
  where $(P_{k})$ is the sequence of Pell numbers,
  \begin{equation*}
    P_{0}\;=\;0,\quad P_{1}\;=\;1,\quad P_{k+1}\;=\;2\,P_{k}\,+\,P_{k-1}.
  \end{equation*}
 \end{remark}

\begin{theorem}[Best approximants of companions]
  \label{thm:companion_approximation}
  The best approximants of a companion
  $\gamma^{\pm}_{k}\big(\frac{p}{q}\big)$, $k\geq 2$, are precisely
  $\frac{p}{q}$ and $\gamma^{\pm}_{k-1}\big(\frac{p}{q}\big)$, i.e.,
  \begin{equation}
    \label{eq:companion_equal_quality}
    C
    \left(
      \gamma^{\pm}_{k}\Big(\frac{p}{q}\Big)
    \right)
    \;=\;
    q^{2}\,
    \left|\,
      \gamma^{\pm}_{k}\Big(\frac{p}{q}\Big)\,-\,\frac{p}{q}\,
    \right|
    \;=\;
    (q\,u_{k-1})^{2}\,
    \left|\,
      \gamma^{\pm}_{k}
      \Big(\frac{p}{q}\Big)
      \,-\,
      \gamma^{\pm}_{k-1}
      \Big(\frac{p}{q}\Big)\,
    \right|
  \end{equation}
  and 
  \begin{equation*}
    C
    \left(
      \gamma^{\pm}_{k}\Big(\frac{p}{q}\Big)
    \right)
    \;<\;
    b^{2}\,
    \left|\,
      \gamma^{\pm}_{k}\Big(\frac{p}{q}\Big)\,-\,\frac{a}{b}\,
    \right|
  \end{equation*}
  for every rational number $\frac{a}{b}$ except $\frac{p}{q}$,
  $\gamma^{\pm}_{k-1}\big(\frac{p}{q}\big)$, and
  $\gamma^{\pm}_{k}\big(\frac{p}{q}\big)$. Moreover,
  \begin{equation}
    \label{eq:Cgamma}
    C\left(
      \gamma^{\pm}_{k}\,
      \Big(\frac{p}{q}\Big)
    \right)
    \;=\;
    q\,\frac{u_{k-1}}{u_{k}}
    \;\geq\;
    \frac{1}{3}\,,
  \end{equation}
  where equality holds if and only if $k=2$, i.e., for first
  companions.
\end{theorem}

For a proof, see Section~\ref{sec:pf_companion_approximants}. The
first companions $\gamma_{2}^{\pm}(\frac{p}{q})$ with approximation
constant $C=\frac{1}{3}$ are therefore the best of the worst
approximable rational numbers. They are also distinguished by the
property that the approximation constant is attained for one unique
best approximant, $\frac{p}{q}=\gamma_{1}^{\pm}\big(\frac{p}{q}\big)$,
instead of two.

\begin{examples}
  \textbf{(i)} The first right and left companions of the Markov
  fraction~$\frac{1}{2}$ are $\gamma^{+}_{2}(\frac{1}{2})=\frac{7}{12}$
  (compare Table~\ref{tab:companions}) and
  $\gamma^{-}_{2}(\frac{1}{2})=\frac{5}{12}$. They are best
  approximated by their associated Markov fraction $\frac{1}{2}$ and
  no other rational number. The approximation constant is
  $C(\frac{7}{12})=C(\frac{5}{12})=\frac{1}{3}$.

  \textbf{(ii)} The second right and left companions of $\frac{1}{2}$
  are $\gamma^{+}_{3}(\frac{1}{2})=\frac{41}{70}$ (compare
  Table~\ref{tab:companions}) and
  $\gamma^{-}_{3}(\frac{1}{2})=\frac{29}{70}$. They are best
  approximated by exactly two rational numbers: their associated
  Markov fraction $\frac{1}{2}$ and its first right and left
  companion, respectively. The approximation number is
  $C(\frac{41}{70})=C(\frac{29}{70})=\frac{12}{35}$.
\end{examples}

\section{The geometric point of view}
\label{sec:geometric_pov}

\subsection{The modular torus}
\label{sec:modular_torus}

Let $G<\PSLTZ$ be the commutator subgroup of the modular group
$\PSLTZ$. The group $G$ is freely generated by the two generators
\begin{equation*}
  [A]\;=\;
  \begin{bmatrix}
    1 & -1 \\ -1 & 2
  \end{bmatrix}
  \quad\text{and}\quad
  [B]\;=\;
  \begin{bmatrix}
    1 & 1 \\ 1 & 2
  \end{bmatrix}.
\end{equation*}
It is a normal subgroup of $\PSLTZ$ with index $6$, and it is the only
subgroup that has a once punctured torus as orbit space. This orbit
space $M\;=\;G\backslash H^{2}$ with canonical projection
\begin{equation*}
  \pi\,:\,H^{2}\,\longrightarrow\,M,\quad \pi(z)\;=\;G\, z\,.
\end{equation*}
is called the \emph{modular torus}.
The ideal quadrilateral with vertices $-1$, $0$, $1$, and $\sminfty$
is a fundamental domain, from which one obtains the torus $M$ by
gluing the pairs of opposite sides via the isometries~$[A]$
and~$[B]$, respectively (see
Figure~\ref{fig:fundamental_domain_1-1-1}).
\begin{figure}
  \centering
  \input{fundamental_domain_1-1-1_t}
  \caption{Fundamental domain of the modular torus with gluing
    transformations}
  \label{fig:fundamental_domain_1-1-1}
\end{figure}

The Farey triangulation of the projective plane~$H^{2}$ projects to
the most symmetric ideal triangulation of~$M$ (see
Figure~\ref{fig:farey_triangulation}). The Ford circles (the unshaded
circles in Figure~\ref{fig:ford}) project to the
boundary of the largest embedded cusp neighborhood in~$M$.

\begin{figure}
  \centering
  \input{farey_triangulation_t}
  \caption{The Farey triangulation. Any two adjacent triangles form a
    fundamental domain of $G$. Edges labeled with the same letter
    project to the same edge in the modular torus~$M$. An element of
    $\PSLTZ$ belongs to $G$ if and only if it respects edge labels and
    triangle shadings. An element of $\PSLTZ$ corresponds to the
    hyperelliptic involution of~$M$ if and only if it respects edge
    labels but maps shaded triangles to unshaded ones and vice versa.}
  \label{fig:farey_triangulation}
\end{figure}

\subsection{The geometric characterization of rational Markov triples and
  Markov fractions}
\label{sec:markov_frac_geometric}

Rational Markov triples correspond to the fundamental domains of the
modular torus~$M$ that are ideal quadrilaterals with one vertex
at~$\sminfty$. This is illustrated in
Figure~\ref{fig:fundamental_domain_1-1-1} for the triple $(-1,0,1)$,
in Figure~\ref{fig:fundamental_domain_1-5-2}\
\begin{figure}
  \centering
  \input{fundamental_domain_1-5-2_t}
  \bigskip
  \input{fundamental_domain_5-29-2_t}
  \caption{Fundamental domains corresponding to the centered rational
    Markov triples $\big(\frac{0}{1}, \frac{2}{5}, \frac{1}{2}\big)$
    and $\big(\frac{2}{5},\frac{12}{29},\frac{1}{2}\big)$ of
    Example~\ref{exa:optimal}. The approximation constants of
    $\frac{2}{5}$ and $\frac{12}{29}$ are
    $e^{\frac{1}{2}d}=C\big(\frac{2}{5}\big)=\frac{2}{5}$ and
    $e^{\frac{1}{2}d'}=C\big(\frac{12}{29}\big)=\frac{10}{29}$,
    respectively (see Lemma~\ref{lem:d_vert_horo} and
    Corollary~\ref{cor:C_geometric}). The signed distances $d$ and
    $d'$ are negative because the vertical geodesics at $\frac{2}{5}$
    and $\frac{12}{29}$ intersect the Ford circles at the best
    approximants.}
  \label{fig:fundamental_domain_1-5-2}
\end{figure}
for the triples $(\frac{0}{1},\frac{2}{5},\frac{1}{2})$ and
$(\frac{2}{5},\frac{12}{29},\frac{1}{2})$, and schematically in
Figure~\ref{fig:fundamental_domain_schematic}.
\begin{figure}
  \centering
  \input{fundamental_domain_schematic_t}
  \caption{If
    $(\frac{p_{1}}{q_{1}},\frac{p_{2}}{q_{2}},\frac{p_{3}}{q_{3}})$ is
    a rational Markov triple, then the ideal quadrilateral with
    vertices $\frac{p_{1}}{q_{1}}$, $\frac{p_{2}}{q_{2}}$,
    $\frac{p_{3}}{q_{3}}$, and $\sminfty$ is a fundamental domain of
    the modular torus. \emph{Dashed:} If we cut along the diagonal
    $q_{2}$ and glue the left ideal triangle
    $\frac{p_{1}}{q_{1}},\frac{p_{2}}{q_{2}},\sminfty$ back to the
    right ideal triangle
    $\frac{p_{2}}{q_{2}},\frac{p_{3}}{q_{3}},\sminfty$, but along edge
    $q_{3}$, we obtain another fundamental domain, which corresponds
    to the rational Markov triple
    $(\frac{p_{2}}{q_{2}},\frac{p_{3}}{q_{3}},\frac{p_{1}}{q_{1}}+3)$
    (compare equation~\eqref{eq:Zshift}).}
  \label{fig:fundamental_domain_schematic}
\end{figure}

\begin{remark}
  \label{rem:fake_geometry}%
  For the sake of clearer pictures,
  Figures~\ref{fig:fundamental_domain_schematic},
  \ref{fig:fundomain_child}, and~\ref{fig:companions_geometric}
  do not show the correct geometry of the modular torus~$M$ but some
  other hyperbolic tori with one cusp. As a result, the objects in
  these figures are more evenly sized and intersect each other less
  than would be the case in the modular torus.
\end{remark}

To put the correspondence between rational Markov triples and
fundamental domains more precisely, we formulate the following lemma:

\begin{lemma}[Geometric characterization of rational Markov triples]
  \label{lem:Markov_triple_geom}
  The following statements for real numbers
  $x_{1}<x_{2}<x_{3}$ are equivalent:

  \begin{compactenum}[(i)]
    \item $(x_{1},x_{2},x_{3})$ is a rational Markov triple.
    \item The ideal quadrilateral with vertices
      $x_{1},x_{2},x_{3},\sminfty$ is a fundamental domain of the
      modular torus~$M$.
    \item The ideal triangles with vertices $x_{1},x_{2},\sminfty$ and
      $x_{2},x_{3},\sminfty$, respectively, project to the two
      triangles of some ideal triangulation of the modular torus~$M$.
  \end{compactenum}
\end{lemma}

(For a proof, see Section~\ref{sec:pr_lem_Markov_geo}.)

Since every simple geodesic with both ends in the cusp is an edge of
some ideal triangulation of the modular torus~$M$, we get the
following characterization of Markov fractions:

\begin{corollary}[Geometric characterization of Markov fractions]
  \label{cor:Markov_frac_geom}
  For a real number $x\in\R$, the following statements are equivalent:
  \begin{compactenum}[(i)]
  \item $x$ is a Markov fraction.
  \item The vertical geodesic $g_{x}$ joining $x$ and $\sminfty$ in $H^{2}$
    projects to a simple geodesic $\pi(g_{x})$ in the modular torus
    $M$ with both ends in the cusp.
  \end{compactenum}
\end{corollary}

\subsection{Proof of Lemma~\ref{lem:Markov_triple_geom} (Geometric characterization of rational Markov triples)}
\label{sec:pr_lem_Markov_geo}

The equivalence of statements (ii) and (iii) of the lemma is
obvious. It remains to show that (i) and~(iii) are equivalent.

\paragraph{(iii)$\Rightarrow$(i).} Assuming (iii) implies that the
vertices $x_{k}\in\R$ are rational,
\begin{equation*}
  x_{k}\;=\;\frac{p_{k}}{q_{k}}\;\in\Q\,,
\end{equation*}
because they project to the ideal point of the modular torus.

The denominators~$q_{k}$ are equal to the weights of the edges of the
triangulation \cite[Ch.~11]{Springborn_Markov} (see
Figure~\ref{fig:fundamental_domain_schematic}). Indeed, the euclidean
diameter of the Ford circle at $\frac{p_{k}}{q_{k}}$ is $q_{k}^{-2}$,
so the distance to the horocycle $\{\im z=1\}$ is
$d_{k}=2\log q_{k}$, and the weight is
$q_{k}=e^{d_{k}/2}$.

By Penner's geometric interpretation of Markov's
equation~\cite{penner87,penner12}, the denominators form a Markov
triple $(q_{1},q_{2},q_{3})$. Finally, the relation between the edge
weights and the horocyclic arcs of a decorated ideal triangle (see
Figure~\ref{fig:ideal_triang})
\begin{figure}
  \begin{minipage}{0.49\textwidth}
    \centering
    \begingroup
\def \globalscale {1.000000}
\begin{tikzpicture}[y=0.80pt, x=0.80pt, yscale=-\globalscale, xscale=\globalscale, inner sep=0pt, outer sep=0pt]
  \begin{scope}[shift={(-268.39,-341.718)}]
    \path[shift={(268.39,341.718)},draw=black,fill=black!10,line cap=rect,line
    join=bevel,line width=0.500pt,miter limit=4.00]
    (78.8672,0.9785) --
    (78.8184,1.0605)arc(0.039:59.987:136.044590 and 136.045) --
    (11.0527,118.7715)arc(240.094:299.884:136.038910 and 136.039) --
    (146.9629,118.7461)arc(120.001:179.923:135.985530 and 135.986) -- cycle;
    
    \path[draw=black,line cap=butt,line join=bevel,line width=0.500pt,miter
    limit=4.00] (347.2577,421.2515) ellipse (2.2170cm and 2.2170cm);
    
    \path[color=black,draw=black,line cap=butt,line join=bevel,line
    width=0.500pt,miter limit=4.00] (294.3342,451.7493) ellipse (0.4929cm and
    0.4929cm);
    
    \path[color=black,draw=black,line cap=butt,line join=bevel,line
    width=0.500pt,miter limit=4.00] (347.3597,371.5814) ellipse (0.8152cm and
    0.8152cm);
    
    \path[color=black,draw=black,line cap=butt,line join=bevel,line
    width=0.500pt,miter limit=4.00] (401.6195,452.6762) ellipse (0.4461cm and
    0.4461cm);
    
    \path[color=black,draw=black,line cap=butt,line join=bevel,line
    width=0.800pt,miter limit=4.00] (279.2122,460.4881) ellipse (0.0135cm and
    0.0135cm);
    
    \path[color=black,draw=black,line cap=butt,line join=bevel,line
    width=0.800pt,miter limit=4.00] (347.2577,342.6965) ellipse (0.0135cm and
    0.0135cm);
    
    \path[color=black,draw=black,line cap=butt,line join=bevel,line
    width=0.800pt,miter limit=4.00] (415.3525,460.4635) ellipse (0.0135cm and
    0.0135cm);
    
    \path[color=black,draw=black,line cap=butt,line join=bevel,line
    width=0.800pt,miter limit=4.00] (311.2088,447.2282) ellipse (0.0135cm and
    0.0135cm);
    
    \path[color=black,draw=black,line cap=butt,line join=bevel,line
    width=0.800pt,miter limit=4.00] (306.7663,439.4821) ellipse (0.0135cm and
    0.0135cm);
    
    \path[color=black,draw=black,line cap=butt,line join=bevel,line
    width=0.800pt,miter limit=4.00] (335.4945,397.9246) ellipse (0.0135cm and
    0.0135cm);
    
    \path[color=black,draw=black,line cap=butt,line join=bevel,line
    width=0.800pt,miter limit=4.00] (359.1031,397.9793) ellipse (0.0135cm and
    0.0135cm);
    
    \path[color=black,draw=black,line cap=butt,line join=bevel,line
    width=0.800pt,miter limit=4.00] (390.1848,441.7627) ellipse (0.0135cm and
    0.0135cm);

    \path[color=black,draw=black,line cap=butt,line join=bevel,line
    width=0.800pt,miter limit=4.00] (386.4734,448.1397) ellipse (0.0135cm and
    0.0135cm);
    
    \path[draw=black,line cap=butt,line join=bevel,line width=1.000pt,miter
    limit=4.00] (335.5023,397.9052)arc(23.949:45.289:136.044590 and 136.045)
    node[midway, above left=3pt]{\small $q_{1}$};
    
    \path[color=black,draw=black,line cap=butt,line join=bevel,line
    width=1.000pt,miter limit=4.00]
    (311.3172,447.2017)arc(254.681:286.626:136.038910 and 136.039)
    node[midway, below=4pt]{\small $q_{3}$};
    
    \path[color=black,draw=black,line cap=butt,line join=bevel,line
    width=1.000pt,miter limit=4.00]
    (390.1180,441.6953)arc(133.280:155.987:135.985530 and 135.986)
    node[midway, above right=3pt]{\small $q_{2}$};

    \path[color=black,draw=black,line cap=butt,line join=bevel,line
    width=1.000pt,miter limit=4.00]
    (358.9509,398.0386)arc(66.341:114.644:28.884882 and 28.885)
    node[midway, below=3pt]{$\frac{q_{3}}{q_{1}q_{2}}$};
  
  

  \end{scope}
\end{tikzpicture}
\endgroup

  \end{minipage}
  \hfill
  \begin{minipage}{0.48\textwidth}
    \centering
    \input{fundamental_domain_children_t}
  \end{minipage}
  \\
  \begin{minipage}[t]{0.47\textwidth}
    \caption{Ideal triangle decorated with horocycles. Relation between
      edge weights~$q_{k}$ and horocyclic arcs}
    \label{fig:ideal_triang}
  \end{minipage}
  \hfill
  \begin{minipage}[t]{0.47\textwidth}
    \caption{Fundamental domain consisting of ideal triangles
      $\frac{p_{1}}{q_{1}},\frac{p_{3}'}{q_{3}'},\sminfty$ and
      $\frac{p_{3}'}{q_{3}'},\frac{p_{2}}{q_{2}},\sminfty$}
    \label{fig:fundomain_child}
  \end{minipage}  
\end{figure}
implies equations~\eqref{eq:p1p2_alt}. Therefore,
$(x_{1},x_{2},x_{3})=\big(\frac{p_{1}}{q_{1}},\frac{p_{2}}{q_{2}},\frac{p_{3}}{q_{3}}\big)$
is a rational Markov triple.

\paragraph{(i)$\Rightarrow$(iii).} The implication is true for any
integer rational Markov triple, i.e., if
$(x_{1},x_{2},x_{3})=(n-1,n,n+1)$ for some $n\in\Z$. We will show: If
the implication is true for a rational Markov triple
$\big(\frac{p_{1}}{q_{1}},\frac{p_{2}}{q_{2}},\frac{p_{3}}{q_{3}}\big)$,
then it also holds for the adjacent triples~\eqref{eq:Zshift} and the
children~\eqref{eq:children}. Since every centered rational Markov
triple is a descendant of an integer triple, and every non-centered
rational Markov triple is adjacent to a centered one, this proves the
implication (i)$\Rightarrow$(iii) for all rational Markov triples.

So let
$\big(\frac{p_{1}}{q_{1}},\frac{p_{2}}{q_{2}},\frac{p_{3}}{q_{3}}\big)$
be a rational Markov triple for which the implication
(i)$\Rightarrow$(iii) holds. To see that the same is true for the
adjacent triples~\eqref{eq:Zshift}, note that $z\mapsto z+6$ is a deck
transformation of the modular torus, and $z\mapsto z+3$ is a lift of
the hyperelliptic involution.

It remains to show that the implication (i)$\Rightarrow$(iii) also
holds for the children~\eqref{eq:children}. We will present the
argument for the left child
$\big(\frac{p_{1}}{q_{1}},\frac{p_{3}'}{q_{3}'},\frac{p_{2}}{q_{2}}\big)$,
see Figure~\ref{fig:fundomain_child}. The argument for the right
child is analogous.

First we choose a matrix $S\in\SLTR$ that satisfies
\begin{equation*}
  S\;
  \begin{pmatrix}
    p_{1} \\ q_{1}
  \end{pmatrix}
  \;=\;
  \pm
  \begin{pmatrix}
    p_{2} \\ q_{2}
  \end{pmatrix}
  \qquad\text{and}\qquad
  S\;
  \begin{pmatrix}
    p_{2} \\ q_{2}
  \end{pmatrix}
  \;=\;
  \pm
  \begin{pmatrix}
    p_{1} \\ q_{1}
  \end{pmatrix},
\end{equation*}
which determines $S$ up to sign. Let us choose
\begin{equation*}
  \begin{split}
    S
    &\;=\;
    \begin{pmatrix}
      p_{2} & p_{1} \\ q_{2} & q_{1}
    \end{pmatrix}
    \cdot
    \begin{pmatrix}
      -p_{1} & p_{2} \\ -q_{1} & q_{2}
    \end{pmatrix}^{-1}
    \\
    &\;=\,
    \frac{1}{p_{2}q_{1}-p_{1}q_{2}}\;
    \begin{pmatrix}
      p_{1}q_{1}+p_{2}q_{2} & -p_{1}^{2}-p_{2}^{2} \\
      q_{1}^{2}+q_{2}^{2} & -p_{1}q_{1}-p_{2}q_{2}
    \end{pmatrix}.
  \end{split}
\end{equation*}
Then the hyperbolic isometry $[S]\in\PSLTR$ maps the Ford circle
$h(p_{1},q_{1})$ at $\frac{p_{1}}{q_{1}}$ to the Ford circle
$h(p_{2},q_{2})$ at $\frac{p_{2}}{q_{2}}$ and vice versa (see
Section~\ref{sec:overview_hyperbolic}). This implies that
$[S]\in\PSLTR$ is a lift of the hyperelliptic involution of the
modular torus $M$, and therefore $S\in\SLTZ$.

Since the ideal triangle with vertices $\frac{p_{1}}{q_{1}}$,
$\frac{p_{2}}{q_{2}}$, $\sminfty$ is a lift of one triangle of an
ideal triangulation of the modular torus, the ideal triangle with
vertices
\begin{equation*}
[S]\cdot\frac{p_{1}}{q_{1}}\;=\;\frac{p_{2}}{q_{2}},\qquad
[S]\cdot\frac{p_{2}}{q_{2}}\;=\;\frac{p_{1}}{q_{1}},\qquad
\text{and}\qquad
[S]\cdot\sminfty
\end{equation*}
is a lift of the other triangle of the triangulation. (Here
$\;\cdot\;$ denotes the isometric action of $\PSLTR$ on $H^{2}$\!.\,)
Applying an edge flip, we obtain a triangulation of the modular torus
whose triangles lift to the ideal triangles with vertices
\begin{equation*}
  \frac{p_{1}}{q_{1}}\,,\quad
  S\cdot\sminfty,\quad
  \sminfty
  \qquad
  \text{and}
  \qquad
  S\cdot\sminfty,\quad
  \frac{p_{2}}{q_{2}}\,,  \quad
  \sminfty,
\end{equation*}
respectively.  Finally, using
equation~\eqref{eq:p1p2} and $(q_{1}^{2}+q_{2}^{2})/q_{3}=q_{3}'$, we
obtain
\begin{equation*}
  S\;
  \begin{pmatrix}
    1 \\ 0
  \end{pmatrix}
  \;=\;
  \begin{pmatrix}
    p_{3}'\\q_{3}'
  \end{pmatrix},
  \quad\text{so}\quad
  S\cdot\sminfty=\frac{p_{3}'}{q_{3}'}
\end{equation*}
and the implication (i)$\Rightarrow$(iii) holds for the child
$\big(\frac{p_{1}}{q_{1}},\frac{p_{3}'}{q_{3}'},\frac{p_{2}}{q_{2}}\big)$,
too. This completes the proof of Lemma~\ref{lem:Markov_triple_geom}.\qed

\subsection{The geometric characterization of companions}
\label{sec:companions_geom}

Figures~\ref{fig:torus3d} and~\ref{fig:companions_geometric}
illustrate the geometric interpretation of the the first right
companion~$\gamma^{+}_{2}(x)$.
\begin{figure}
  \centering
  \input{torus3d_t}
  \caption{Topological sketch of the modular torus $M$ showing the
    geodesic $\pi(g_{x})$ corresponding to a Markov fraction $x$, and
    the simple closed geodesic $\eta$ that does not intersect it. The
    geodesic $\pi(g_{\gamma_{2}^{+}(x)})$ corresponds to the first
    right companion $\gamma_{2}^{+}(x)$. It is contained in the
    cylinder $\pi(R_{x})$ bounded by $\pi(g_{x})$ and $\eta$ and has
    one self-intersection.}
  \label{fig:torus3d}
\end{figure}

Generally, the vertical geodesic $g_{x}$ projects to a simple
geodesic~$\pi(g_{x})$ in the Modular torus~$M$ with both ends in the
cusp (by Corollary~\ref{cor:Markov_frac_geom}). Then there is a
unique simple closed geodesic $\eta$ in $M$ that does not
intersect~$\pi(g_{x})$. The geodesics $\pi(g_{x})$ and $\eta$ separate
$M$ into two cylinders, which are labeled $\pi(R_{x})$ and
$\pi(L_{x})$ in Figure~\ref{fig:torus3d} for reasons that will become
clear in the following Section~\ref{sec:pf_companions_geom}. Be that
as it may, $\pi(R_{x})$ is the cylinder to the right of $\pi(g_{x})$
if $g_{x}$ is oriented in the direction from $x$ to $\sminfty$, and
$\pi(L_{x})$ is the cylinder to the left.
 
Now let $y$ be a right or left companion of $x$, i.e.,
$y=\gamma^{+}_{k}(x)$ or $y=\gamma^{-}_{k}(x)$ for some $k\geq 2$.
Then the vertical geodesic $g_{y}$ projects to a non-simple geodesic
$\pi(g_{y})$ in the cylinder $\pi(R_{x})$ or $\pi(L_{x})$,
respectively, with $k-1$ self-intersections.

Formulating the converse statement is slightly more involved. To
understand why, note first that the isometry $z\mapsto z+6$ of $H^{2}$
is a deck transformation of~$M$ (see
Figure~\ref{fig:farey_triangulation}). Hence the companions of $x$
correspond to the same vertical geodesics in $M$ as the companions of
$x+6n$ for any $n\in\Z$. Similarly, the isometry $z\mapsto z+3$
of~$H^{2}$ projects to the hyperelliptic involution of $M$, which
reverses the orientations of $\pi(g_{x})$ and $\eta$. As a
consequence, the cylinder $\pi(R_{x})$ contains not only the projected
vertical geodesics for the right companions $\gamma^{+}_{k}(x)$, but
also those for $\gamma^{-}_{k}(x+3)$, the left companions of $x+3$
(see Figure~\ref{fig:companions_geometric} and compare
Remark~\ref{rem:companions_symmetries}).

Altogether, the following lemma characterizes the right and left
companions of a Markov fraction geometrically:

\begin{lemma}[Geometric characterization of companions]
  \label{lem:companions_geom}
  Let $x$ be a Markov fraction, so that $\pi(g_{x})$ is a simple
  geodesic in the modular torus $M$ with both ends in the cusp. Let
  $\eta$ be the unique simple closed geodesic in $M$ that does not
  intersect $\pi(g_{x})$. Then for a number $y\in\R$ and the vertical
  geodesic $g_{y}$, the statements (i$^+$) and (ii$^+$) are
  equivalent, and the statements (i$^-$) and (ii$^-$) are equivalent:
  
  \medskip%
  \begin{compactenum}[(i$^+$)]
  \item $y$ is a right companion of $x$, i.e., $y=\gamma^{+}_{k}(x)$ for some $k\geq 2$.
  \item $y\in[x,\, x+\frac{3}{2}]$ and $\pi(g_{y})$ is a geodesic with
    both ends in the cusp that intersects neither $\pi(g_{x})$ nor
    $\eta$.
  \end{compactenum}
  \medskip%
  \begin{compactenum}[(i$^-$)]
  \item $y$ is a left companion of $x$, i.e., $y=\gamma^{-}_{k}(x)$ for some $k\geq 2$.
  \item $y\in[x-\frac{3}{2}, \,x]$ and $\pi(g_{y})$ is a geodesic with
    both ends in the cusp that intersects neither $\pi(g_{x})$ nor
    $\eta$.
  \end{compactenum}
  \medskip%
  The geodesics $\pi(g_{\gamma^{+}_{k}(x)})$ and
  $\pi(g_{\gamma^{-}_{k}(x)})$ have $k-1$ self-intersections.
\end{lemma}

\subsection{Proof of Lemma~\ref{lem:companions_geom} (Geometric
  characterization of companions)}
\label{sec:pf_companions_geom}

The geodesics in the $G$-orbit $G\cdot g_{x}$ of the vertical
geodesic~$g_{x}$ together with~$\pi^{-1}(\eta)$, the $G$-orbit of
lifts of the closed geodesic~$\eta$, separate the hyperbolic
plane~$H^{2}$ into simply connected domains. Each one of these domains
covers one of the two cylinders into which~$\pi(g_{x})$ and~$\eta$
separate~$M$. Of these domains, let~$R_{x}$ and~$L_{x}$ be the ones to
the right and to the left of~$g_{x}$, respectively, when~$g_{x}$ is
oriented in the direction from~$x$ to $\sminfty$ (see
Figure~\ref{fig:companions_geometric}).
\begin{figure}
  \centering
  \input{companions_geometric_t}
  \caption{The right companions $\gamma^{+}_{2}(x),\,\ldots\,$ are
    ideal vertices of region $R_{x}$ between $x$ and $x+\frac{3}{2}$.
    The left companions $\gamma^{-}_{2}(x),\,\ldots\,$ are ideal
    vertices of the region $L_{x}$ between $x-\frac{3}{2}$ and
    $x$. The dotted lines are the geodesics
    $[T]^{k}\cdot g_{\gamma^{+}_{2}(x)}$.}
  \label{fig:companions_geometric}
\end{figure}
In short, $R_{x}$ and~$L_{x}$ are the domains to the left and right
of~$g_{x}$ such that the restrictions~$\pi|_{R_{x}}$
and~$\pi|_{L_{x}}$ are universal covers of the two cylinders in $M$.

The geometric characterization of the companions $\gamma^{\pm}_{k}(x)$
follows in a straightforward manner from an explicit analytic
description of the domains~$R_{x}$ and~$L_{x}$. We will perform the
calculations only for~$R_{x}$ and the right companions
$\gamma^{+}_{k}(x)$. The domain~$L_{x}$ and the left companions
$\gamma^{-}_{k}(x)$ can be treated analogously. Alternatively, one may
use the fact that $z\mapsto z+3$ projects to the hyperelliptic
involution to see that $L_{x}=R_{x}-3$ (compare
Remark~\ref{rem:companions_symmetries}).

In order to describe the boundary of the domain $R_{x}$, note first
that the cylinder $\pi(R_{x})\subset M$ has two geodesic boundary
components. One is the closed geodesic~$\eta$, and the other is the
geodesic $\pi(g_{x})$ with both ends in the cusp of $M$. Accordingly,
the boundary of the universal cover $R_{x}$ of this cylinder consists
of two parts (compare Figure~\ref{fig:companions_geometric}):
\begin{compactenum}[(1)]
\item a hyperbolic line $\varrho_{x}$ that is the universal cover of
  the closed geodesic $\eta$, and
\item a doubly infinite sequence of hyperbolic lines in the $G$-orbit
  of $g_{x}$ including~$g_{x}$ itself, that form a doubly infinite
  ideal polygon.
\end{compactenum}
Since $\pi(R_{x})$ is a cylinder with fundamental group $\Z$, the group
of deck transformations of the universal cover
$R_{x}\rightarrow\pi(R_{x})$ is generated by a hyperbolic isometry
\begin{equation*}
  [T]\;\in\; G\;\subset\;\PSLTZ,
\end{equation*}
which maps the hyperbolic line $\varrho_{x}$ to itself and each edge
of the ideal polygon (2) to the next. In particular, $[T]$ is a
hyperbolic translation with axis $\varrho_{x}$, the edges of the ideal
polygon are $[T]^{k}\cdot g_{x}$ for $k\in\Z$, and therefore the ideal
vertices of the polygonal boundary component accumulate at the
endpoints of the axis $\varrho_{x}$.

One geodesic in the polygonal boundary of $R_{x}$ is known by
definition, the vertical geodesic $g_{x}$.  Another one is the vertical
geodesic $g_{x+3}$. To see this, note that since $\pi(g_{x})$ is
simple geodesic in the modular torus $M$ with both ends in the cusp,
there is an ideal triangulation of $M$ consisting of $\pi(g_{x})$ and
two other such geodesics. Now contemplating
Figure~\ref{fig:fundamental_domain_schematic} with
$x=\frac{p_{1}}{q_{q}}$, one sees that $g_{x}$ and $g_{x+3}$ are
adjacent edges in the ideal polygonal boundary of $R_{x}$ sharing the
common vertex $\infty$.

To derive an explicit expression for the generator $[T]\in\PSLTZ$, we
let
\begin{equation*}
x=\frac{p}{q},
\end{equation*}
and we may assume that $[T]$ maps $g_{x+3}$ to $g_{x}$. (The the other
generator is $[T]^{-1}$, which maps $g_{x}$ to $g_{x+3}$.) This means
not only that $[T]$ maps the endpoints of $g_{x+3}$ ($x+3$
and $\sminfty$) to the corresponding endpoints  of
$g_{x}$ ($\infty$ and $x$). But the isometry~$[T]$ also maps the Ford circles at these
corresponding endpoints to each other, i.e., $h(p+3q, q)$ to $h(1,0)$
and $h(1,0)$ to $h(p,q)$, which determines the matrix $T\in\SLTZ$ up
to sign~(see Section~\ref{sec:overview_hyperbolic} and
\cite[Ch.~5]{Springborn_Markov}). Choosing
\begin{equation*}
  T\;
  \begin{pmatrix}
    1 & p+3\,q \\
    0 & q
  \end{pmatrix}
  \;=\;
  \begin{pmatrix}
    -p & 1 \\
    -q & 0
  \end{pmatrix},
\end{equation*}
which works because
\begin{equation*}
  \det\,
  \begin{pmatrix}
    1 & p+3\,q \\
    0 & q
  \end{pmatrix}
  \;=\;q\;=\;
  \det\,
  \begin{pmatrix}
    -p & 1 \\
    -q & 0
  \end{pmatrix}
\end{equation*}
we obtain
\begin{equation}
  \label{eq:T}
  T\;=\;
  \begin{pmatrix}
    -p & 3\,p+\frac{p^{2}+1}{q} \\
    -q & 3\,q+p
  \end{pmatrix}
  \;\in\;\SLTZ\,.
\end{equation}

\begin{remark}
  In particular, this implies $\frac{p^{2}+1}{q}\in\Z$, which provides
  a geometric interpretation of the number theoretic
  Corollary~\ref{cor:quadratic_residue}. More directly, such a
  geometric interpretation can be obtained by considering the
  orientation preserving hyperbolic isometry $[S_{p,q}]\in\PSLTR$ that
  interchanges the horocycles $h(p,q)$ and $h(1,0)$.
  This determines
  \begin{equation*}
    S_{p,q}\;=\;
    \begin{pmatrix}
     -p & \frac{p^{2}+1}{q} \\
     -q & p     
   \end{pmatrix}
   \;\in\;\SLTR
  \end{equation*}
  up to sign. Since $[S_{p,q}]$ projects to the hyperelliptic
  involution on $M$, we have $S_{p,q}\in\SLTZ$. (The map $[T]$ is the
  composition of $[S_{p,q}]$ followed by $z\mapsto z+3$.)
\end{remark}

The axis $\varrho_{x}$ of the hyperbolic translation $[T]$ connects
the attracting and repelling fixed points of $[T]$. Let us call them
\begin{equation}
  \label{eq:Tk_lim}
  \xi_{+}=\lim_{k\rightarrow+\infty}[T]^{k}\cdot\sminfty
  \quad\text{and}\quad
  \xi_{-}=\lim_{k\rightarrow-\infty}[T]^{k}\cdot\sminfty,
\end{equation}
respectively, where convergence is in the topology of $\C P^{1}$. (We
might equally well let $[T]^{k}$ act on any other point except the
fixed points, but we choose $\sminfty$ so we can re-use
equation~\eqref{eq:Tk_lim} in the proof of Lemma~\ref{lem:Ipq}.) By a
straightforward calculation we obtain
\begin{equation}
  \label{eq:T_fp}
  \xi_{\pm}\;=\;
  x\,+\,\frac{3}{2}
  \,\mp\,\sqrt{\frac{9}{4}\,-\,\frac{1}{q^{2}}}\;.
\end{equation}

\begin{remark}
  \label{rem:len_closed_geodesic}
  Since $\trace(T)=3q$, the eigenvalues of $T$ are
  $\frac{1}{2}\,(3\,q\pm\sqrt{9\,q^{2}-4})$. This also shows the
  well-known formulas for the lengths of closed geodesics in the
  modular torus,
  \begin{equation*}
    2\,\cosh
    \left(
      \frac{\length(\eta)}{2}
    \right)
    \;=\;3\,q\,,
    \qquad
    2\,\exp
    \left(
      \frac{\length(\eta)}{2}
    \right)
    \;=\;
      3\,q\;+\;\sqrt{9\,q^{2}\,-\,4}\,,
  \end{equation*}
  but we will not need them.
\end{remark}

Now the equivalence of the statements (i$^+$) and (ii$^+$) follows
from the fact that
\begin{equation}
  \label{eq:companion_geometric}
  [T]^{k}\cdot \sminfty\;=\;\gamma_{k}^{+}(x)
\end{equation}
(with $\gamma_{k}^{+}(x)$ defined by equation~\eqref{eq:companions}).
To see this, let $(r_{k})_{k\in \Z}$ and $(s_{k})_{k\in \Z}$ be
defined by
\begin{equation}
  \label{eq:rs}
  \begin{pmatrix}
    r_{k} \\ s_{k}
  \end{pmatrix}
  \;=\;
  T^{k}\;
  \begin{pmatrix}
    -1 \\ 0
  \end{pmatrix}\,,
\end{equation}
so that
\begin{equation*}
  [T]^{k}\cdot \sminfty\;=\;\frac{r_{k}}{s_{k}}\;.
\end{equation*}
Then it suffices to show that
\begin{equation}
  \label{eq:rs_of_u}
  \begin{pmatrix}
    r_{k}\\s_{k}
  \end{pmatrix}
  \;=\;
  \begin{pmatrix}
    p\,u_{k}\,+\,u_{k-1}\\q\,u_{k}
  \end{pmatrix}
\end{equation}
for the sequence $(u_{k})_{k\in \Z}$ defined by~\eqref{eq:un}.

First note that $q$ divides $s_{k}$ for all $k\in\Z$. (This is easy to
see by induction in both directions starting from $s_{0}=0$.) Hence we
can \emph{define} integer sequences $(t_{k})_{k\in\Z}$ and
$(u_{k})_{k\in\Z}$ by
\begin{equation*}
  u_{k}\;=\;\frac{s_{k}}{q} \qquad\text{and}\qquad
  t_{k}\;=\;r_{k}-p\,u_{k}.
\end{equation*}
We have to show that $(u_{k})_{k\in\Z}$ solves the
recursion~\eqref{eq:un}. The initial conditions are easily
verified. Using the relation
\begin{equation*}
  \begin{pmatrix}
    r_{k}\\s_{k}
  \end{pmatrix}
  \;=\;
  \underbrace{
    \begin{pmatrix}
      1 & p\\
      0 & q
    \end{pmatrix}
  }_{=:\,U}\;
  \begin{pmatrix}
    t_{k}\\u_{k}
  \end{pmatrix},
\end{equation*}
and
$\big(
\begin{smallmatrix}
  r_{k+1}\\s_{k+1}
\end{smallmatrix}
\big)
=T\,
\big(
\begin{smallmatrix}
  r_{k}\\s_{k}
\end{smallmatrix}
\big)$
we obtain the first order recursion formula
\begin{equation*}
  \begin{pmatrix}
    t_{k+1}\\u_{k+1}
  \end{pmatrix}
  \;=\;
  \underbrace{
    \begin{pmatrix}
      0  & 1 \\
      -1 & 3q
    \end{pmatrix}
  }_{U^{-1}\,T\,U}\;
  \begin{pmatrix}
    t_{k}\\u_{k}
  \end{pmatrix}
  .
\end{equation*}
Eliminate $(t_{k})$ using $t_{k}=u_{k-1}$ to obtain the
second-order recursion~\eqref{eq:un} for $(u_{k})_{k\in\Z}$ and
equation~\eqref{eq:rs_of_u} for $r_{k}$ and $s_{k}$.  This completes
the proof of equation~\eqref{eq:companion_geometric}, showing that
(i$^{+}$) and (ii$^{+}$) are equivalent.

Finally, to see that the geodesic $\pi(g_{\gamma^{+}_{k}(x)})$ in the
cylinder $\pi(R_{x})\subset M$ has $k-1$ self-intersections, note that
the vertical geodesic $g_{\gamma^{+}_{k}(x)}$ in the universal cover
$R_{x}$ intersects the $2\,(k-1)$ geodesics
\begin{equation}
  [T]^{\ell}\cdot g_{\gamma^{+}_{k}(x)}
  \quad\text{for}\quad
  \ell\;\in\;\{1,\ldots,k-1\}\;\cup\;\{-(k-1),\ldots,-1\}\,.
\end{equation}
but the intersection points
\begin{equation}
  \zeta_{\ell}\;=\;g_{\gamma^{+}_{k}(x)}
  \;\cap\;
  \left(
    [T]^{\ell}\cdot g_{\gamma^{+}_{k}(x)}
  \right)
\end{equation}
project in pairs to the same point in $M$,
\begin{equation}
  \pi(\zeta_{\ell})\;=\;\pi(\zeta_{-\ell}),
\end{equation}
because
\begin{equation}
  \zeta_{\ell}\;=\;[T]^{\ell}\cdot\zeta_{-\ell}\,.
\end{equation}
This completes the proof of Lemma~\ref{lem:companions_geom}.\qed

\bigskip%
Because $T\in\SLTZ$, equation~\eqref{eq:rs} implies that $r_{k}$ and
$s_{k}$ are coprime for all $k\in\Z$. So the above proof (together
with the analogous argument for the left companions) yields the
following corollary:

\begin{corollary}
  \label{cor:rs_coprime}
  For all $k\in\Z$, the integers $p\,u_{k}+ u_{k-1}$ and $q\,u_{k}$
  are coprime, and so are the integers $p\,u_{k}-u_{k-1}$ and
  $q\,u_{k}$.
\end{corollary}

\subsection{Proof of Theorem~\ref{thm:Markov_best_approximants} (ii)
  (Best approximants of Markov fractions)}
\label{sec:pr_Markov_approximants}

Part (i) of Theorem~\ref{thm:Markov_best_approximants} was proved
without recourse to geometry in
Section~\ref{sec:pf_thm_unique_centered_triple}. The purpose of this
section is to prove part (ii). The argument begins algebraically, but
geometry will come into play shortly.

Let
\begin{equation*}
  m\;=\;\left(\,\frac{p_{1}}{q_{1}},\;\frac{p_{2}}{q_{2}},\;\frac{p_{3}}{q_{3}}\,\right)
\end{equation*}
be a centered rational Markov triple. Then we have
\begin{equation*}
  q_{1}^{2}\,\cdot\,
  \left(\,
    \frac{p_{2}}{q_{2}}\;-\;\frac{p_{1}}{q_{1}}\,
  \right)
  \;\;\;\overset{\makebox[0pt]{\scriptsize{\eqref{eq:p1p2}}}}{=}\;\;\;
  \frac{q_{1}\,q_{3}}{q_{2}}
  \;\;\;\overset{\makebox[0pt]{\scriptsize{\eqref{eq:p2p3}}}}{=}\;\;\;
  q_{3}^{2}\,\cdot\,
  \left(\,
    \frac{p_{3}}{q_{3}}\;-\;\frac{p_{2}}{q_{2}}\,.
  \right)\,,
\end{equation*}
Moreover,
\begin{equation*}
  \frac{1}{3}\;<\;\frac{q_{1}\,q_{3}}{q_{2}}\;\leq\;1
\end{equation*}
from equation~\eqref{eq:Markov_q}, together with~\eqref{eq:centered}
for the upper bound. It remains to show that for every rational number
$\frac{a}{b}\not\in\big\{\frac{p_{1}}{q_{1}},\frac{p_{2}}{q_{2}},\frac{p_{3}}{q_{3}}\big\}$,
\begin{equation}
  \label{eq:check_ab}
  \frac{q_{1}\,q_{3}}{q_{2}}
  \;<\;
  b^{2}\cdot\,\left|\,\frac{p_{2}}{q_{2}}\;-\;\frac{a}{b}\,\right|\,.
\end{equation}
To see this, we will rely on the geometric characterization of Markov
fractions (see Sections~\ref{sec:companions_geom}
and~\ref{sec:pf_companions_geom}) to narrow down the set of potential
best approximants so we can check~\eqref{eq:check_ab} for the
remaining candidates.

So let $x=\frac{p}{q}$ be a Markov fraction, and let
$x^{*}=\frac{p^{*}}{q^{*}}$ be a best approximant, i.e., suppose
\begin{equation}
  C\left(\frac{p}{q}\right)
  \;=\;(q^{*})^{2}
  \,\cdot\,
  \left|\;\frac{p}{q}\;-\;\frac{p^{*}}{q^{*}}\;\right|\,.
\end{equation}
By Lemma~\ref{lem:d_vert_horo}, this means that the Ford circle
$h(p^{*},q^{*})$ has the smallest signed distance to the vertical
geodesic~$g_{x}$ among all Ford circles except $h(p,q)$ and
$h(1,0)$. This leads to the following necessary geometric condition
for best approximants:

\begin{lemma}
  \label{lem:nearest_to_Markov}
  If~$x$ is Markov fraction and~$x^{*}$ is a best approximant of~$x$,
  then no geodesic in the $G$-orbit of $g_{x}$ separates~$g_{x}$
  and~$x^{*}$.
\end{lemma}

(A geodesic $g$ in the hyperbolic plane $H^{2}$ \emph{separates} a
subset $A\in H^{2}$ and an ideal point $p\in\partial\!H^{2}$ if $A$
and $p$ are contained in different connected components of
$(H^{2}\cup\partial\!H^{2})\setminus(g\cup\partial\!g)$, where
$\partial\!g$ is the set of the two ideal endpoints of $g$.)

\begin{proof}[Proof of Lemma~\ref{lem:nearest_to_Markov}]
  Suppose $\tau\cdot g_{x}$ separates $g_{x}$ and $x'$ for some
  $x'=\frac{p'}{q'}\in\Q$ and $\tau\in G$. We have to show that $x'$
  is not a best approximant of $x$.

  Since the shortest geodesic ray from $x'$ to $g_{x}$ intersects
  $\tau(g_{x})$ before reaching~$g_{x}$,
  \begin{equation*}
    d\big(g_{x},\,h(\,p',q'\,)\big)
    \;>\;d\big(\tau(g_{x}), \,h(\,p',q'\,)\big)
    \;=\;d\big(g_{x},\,\tau^{-1}(\,h(\,p',q'\,)\,)\big).
  \end{equation*}
  Since the Ford circle $\tau^{-1}(h(p',q'))$ has a smaller signed
  distance to $g_{x}$ than $h(p',q')$, the number $\tau^{-1}\cdot x'$
  approximates $x$ better than $x'$.
\end{proof}

Lemma~\ref{lem:nearest_to_Markov} implies that $x^{*}$ is a rational
ideal boundary point of one of the domains
\begin{equation*}
  R_{x}\quad\text{or}\quad L'_{x}
\end{equation*}
that are bounded by the geodesic $\varrho_{x}$ and geodesics in the
$G$-orbit of $g_{x}$ (see Figure~\ref{fig:arc_orbit})
\begin{figure}
  \centering
  \input{arc_orbit_t}
  \caption{The domains $R_{x}$ and $L'_{x}$ bounded by $\varrho_{x}$
    and geodesics in the $G$-orbit of $g_{x}$}
  \label{fig:arc_orbit}
\end{figure}
or one of the domains
\begin{equation*}
L_{x}=R_{x}-3 \quad\text{and}\quad R'_{x}=L'_{x}-3
\end{equation*}
bounded by $\lambda_{x}=\varrho_{x}-3$ and geodesics in the $G$-orbit
of $g_{x}$.

In other words, $L'_{x}$ is the lift of the cylinder $\pi(L_{x})$
adjacent to $\varrho_{x}$, and $R'_{x}$ is the lift of the cylinder
$\pi(R_{x})$ adjacent to the geodesic $\lambda_{x}$. (See also
Figure~\ref{fig:companions_geometric}, which shows $\lambda_{x}$ and
$R'_{x}$. But keep in mind that Figure~\ref{fig:companions_geometric}
does not represent the geometry of the modular torus $M$ correctly
whereas Figure~\ref{fig:arc_orbit} does, as explained in
Remark~\ref{rem:fake_geometry}.)

We will now consider all these rational ideal boundary points that are
candidates for best approximants in turn. First, let us consider the
rational ideal boundary points of $R_{x}$, which are:

\medskip%
\begin{compactitem}
\item the Markov fraction $x=\frac{p}{q}$ itself, which does not
  qualify,
\item the right companions $\gamma^{+}_{k}(x)$, with $k\geq 2$,
\item the left companions $\gamma^{-}_{k}(x+3)$, also with $k\geq 2$, and
\item the Markov fraction $x+3$.
\end{compactitem}

\medskip%
The left companions of $x+3$ are ruled out as best approximants
because $\gamma^{-}_{k}(x+3)$ has the same denominator as
$\gamma^{+}_{k}(x)$ but is further away from~$x$.

For a right companion
\begin{equation*}
  \frac{a}{b}\;=\;\gamma^{+}_{k}\,
  \Big(\,\frac{p}{q}\,\Big)\,,
\end{equation*}
we get using Lemmas~\ref{lem:companion_reduced} and~\ref{lem:u_prop}:
\begin{equation*}
  \begin{split}
    b^{2}\cdot\,\left|\,\frac{p_{2}}{q_{2}}\;-\;\frac{a}{b}\,\right|
    \;\;&=\;\;
    q^{2}\,u_{k}^{2}\;
    \left|\;
      \frac{p}{q}\;-\;
      \gamma^{+}_{k}\,
      \Big(\,\frac{p}{q}\,\Big)\;
    \right|\\
    &\overset{\makebox[0pt]{\scriptsize\eqref{eq:companions}}}{=}\;\;
    q\,u_{k}\,u_{k-1}\\
    &\overset{\makebox[0pt]{\scriptsize\eqref{eq:u_conserved}}}{=}\;\;
      \tfrac{1}{3}\,(\,u_{k}^{2}\,+\,u_{k-1}\,-\,1\,)\\
    &\geq\;\;\tfrac{1}{3}\,(\,u_{2}^{2}\,+\,u_{1}\,-\,1\,)\\
    &\overset{\makebox[0pt]{\scriptsize\eqref{eq:un}}}{=}\;\;
    3\,q^{2}\;\geq\; 3
  \end{split}
\end{equation*}
This means that a right companion $\gamma^{+}_{k}(\frac{p}{q})$
approximates the Markov fraction $\frac{p}{q}$ worse than the Markov
fractions $\frac{p_{1}}{q_{1}}$ and~$\frac{p_{3}}{q_{3}}$ of the
centered rational Markov triple
$\big(\frac{p_{1}}{q_{1}},\frac{p_{2}}{q_{2}},\frac{p_{3}}{q_{3}}\big)$
with $\frac{p_{2}}{q_{2}}=\frac{p}{q}$.

One ideal boundary point of $R_{x}$ is left to check: The Markov
fraction $x+3$, for which we get $1^{2}\big|x-(x+3)\big|=3$.

Thus we conclude that the ideal boundary points of $R_{x}$ are not best
approximants of $x$. Similarly, the ideal boundary points of $L_{x}$
are not best approximants of $x$.

\begin{remark}
  Note that the first companion $\gamma^{+}_{2}(x)$ and the Markov
  fraction $x+3$ approximate the Markov fraction $x$ exactly equally
  badly. In general, the $(k-1)$st right companion $\gamma^{+}_{k}(x)$
  and the $(k-2)$nd left companion $\gamma_{k-1}(x+3)$ approximate~$x$
  with the same approximation quality. This due to a geometric
  symmetry.
\end{remark}

Now let $y$ be a rational ideal boundary point of $L'_{x}$. Then the
vertical geodesic~$g_{y}$ projects to a simple geodesic $\pi(g_{y})$
in $M$ with both ends in the cusp. Furthermore, $\pi(g_{y})$ does not
intersect of the geodesic $\pi(g_{x})$. (To see this, note that
$g_{y}$ does not intersect any of the geodesics $[T]^{k}\cdot g_{y}$
for $k\in\Z\setminus\{0\}$, nor any geodesics in the $G$-orbit of
$g_{x}$.)

This implies, first, that $y$ is also Markov fraction. Moreover, since
$\pi(g_{x})$ and $\pi(g_{y})$ do not intersect, there are exactly two
ideal triangulations of $M$ containing $\pi(g_{x})$ and $\pi(g_{y})$
as edges. Since $x<y<x+3$, we can deduce that there is a unique
rational Markov triple
\begin{equation*}
  m'\;=\;
  \left(\frac{p_{1}'}{q_{1}'},\frac{p_{2}}{q_{2}},\frac{p_{3}'}{q_{3}'}\right)
  \qquad\text{with}\qquad
  x\;=\;\frac{p_{2}}{q_{2}},
  \qquad
  y\;=\;\frac{p_{3}'}{q_{3}'}\,.
\end{equation*}
Conversely, if $m'$ is a rational Markov triple with
$x=\frac{p_{2}}{q_{2}}$, then $\frac{p_{3}'}{q_{3}'}$ is an ideal
boundary point of $L_{x}'$.

For the ideal boundary point $y$ of $L_{x}'$, we get the approximation quality
\begin{equation*}
  (q_{3}')^{2}\cdot\,\left|\,\frac{p_{2}}{q_{2}}\;-\;\frac{p_{3}'}{q_{3}'}\,\right|
  \;=\;
  \frac{q_{3}'\,q_{1}'}{q_{2}}\,,
\end{equation*}
and if $m'$ is not the unique centered rational Markov triple $m$, then
\begin{equation*}
  \frac{q_{3}'\,q_{1}'}{q_{2}}
  \;>\;
  \frac{q_{3}\,q_{1}}{q_{2}}\,.
\end{equation*}
This shows that among all ideal boundary points of $L_{x}'$, the third
element $\frac{p_{3}}{q_{3}}$ of the centered rational Markov triple
$m$ with $x=\frac{p_{2}}{q_{2}}$ approximates $x$ best, and all other
ideal boundary points of $L_{x}'$ are worse.

In the same way, one sees that among all ideal boundary points of $R_{x}'$, the first
element $\frac{p_{1}}{q_{1}}$ of the centered rational Markov triple
$m$ with $x=\frac{p_{2}}{q_{2}}$ approximates~$x$ best, and all other
ideal boundary points of $R_{x}'$ are worse.

This completes the proof of
Theorem~\ref{thm:Markov_best_approximants}~(ii).  \qed

\begin{remark}
  While many constructions and arguments in this paper generalize to
  arbitrary hyperbolic tori with one cusp,
  Theorem~\ref{thm:Markov_best_approximants} (ii) really depends on
  the geometry of the modular torus $M$. For example, in the torus
  shown in Figure~\ref{fig:companions_geometric}, the horocycles at
  $\gamma^{\pm}_{2}(x)$ are closer to $g_{x}$ than the horocycles at
  the ideal boundary points of $L_{x}'$ and $R_{x}'$.
\end{remark}

\subsection{Proof of Theorem~\ref{thm:companion_approximation} (Best
  approximants of companions)}
\label{sec:pf_companion_approximants}

From the definition~\eqref{eq:companions}, we get immediately
\begin{equation*}
  q^{2}\,
  \left|\,
    \gamma^{\pm}_{k}\Big(\frac{p}{q}\Big)\,-\,\frac{p}{q}\,
  \right|
  \;=\;
  q\,\frac{u_{k-1}}{u_{k}}.
\end{equation*}
To see that also
\begin{equation*}
  (q\,u_{k-1})^{2}\,
  \left|\,
    \gamma^{\pm}_{k}
    \Big(\frac{p}{q}\Big)
    \,-\,
    \gamma^{\pm}_{k-1}
    \Big(\frac{p}{q}\Big)\,
  \right|  
  \;=\;
  q\,\frac{u_{k-1}}{u_{k}}\,,
\end{equation*}
note that
\begin{equation}
  \label{eq:companion_gap}
  \left|\,
    \gamma^{\pm}_{k}
    \Big(\frac{p}{q}\Big)
    \,-\,
    \gamma^{\pm}_{k-1}
    \Big(\frac{p}{q}\Big)\,
  \right|
  \;=\;
  \frac{1}{q\,u_{k}\,u_{k-1}}
\end{equation}
because
\begin{equation*}
  \frac{u_{k-1}}{q\,u_{k}}
  \,-\,
  \frac{u_{k-2}}{q\,u_{k-1}}
  \;=\;
  \frac{1}{q\,u_{k}\,u_{k-1}}\;
  \big(\,u_{k-1}^{2}\,-\,u_{k}\,u_{k-2}\,\big)
\end{equation*}
and
\begin{equation*}
  \begin{split}
    u_{k-1}^{2}\,-\,u_{k}\,u_{k-2} \;\;
    &\overset{\makebox[0pt]{\scriptsize\eqref{eq:un}}}{=}
    \;\;
    u_{k-1}^{2}\,-\,(3\,q\,u_{k-1}\,-\,u_{k-2})\,u_{k-2}\\
    &=
    \;\;
    u_{k-1}^{2}\,-\,3\,q\,u_{k-1}\,u_{k-2}\,+\,u_{k-2}^{2}\\
    &\overset{\makebox[0pt]{\scriptsize\eqref{eq:u_conserved}}}{=}
    \;\;
    1\,.
  \end{split}
\end{equation*}
This shows the right equality in~\eqref{eq:companion_equal_quality}.

\begin{remark}
  Alternatively, we could avoid these calculations here and infer from
  a geometric symmetry that $\frac{p}{q}$ and
  $\gamma^{\pm}_{k-1}\big(\frac{p}{q}\big)$ both approximate
  $\gamma^{\pm}_{k}\big(\frac{p}{q}\big)$ with the same quality. But
  we will need equation~\eqref{eq:companion_gap} anyway for the proof
  of the classification Theorem~\ref{thm:classify} (see
  Section~\ref{sec:pf_classify}).
\end{remark}

To see the inequality in~\eqref{eq:Cgamma}, note that
\begin{equation*}
  q\,u_{k-1}
  \;\;
  \overset{\makebox[0pt]{\scriptsize\eqref{eq:un}}}{=}
  \;\;
  \frac{1}{3}\,(\,u_{k}\,+\,u_{k-2}\,)\,.
\end{equation*}
So for $k\geq 2$,
\begin{equation*}
  q\,\frac{u_{k-1}}{u_{k}}
  \;=\;
  \frac{1}{3}
  \left(
    1\;+\;\frac{u_{k-2}}{u_{k}}
  \right)
  \;\geq\;
  \frac{1}{3}\;,
\end{equation*}
with equality only for $k=2$.

It remains to show that
\begin{equation}
  \label{eq:companion_candidate_test}
  q\,\frac{u_{k-1}}{u_{k}}
  \;<\;
  b^{2}\,
  \left|\,
    \gamma^{\pm}_{k}\Big(\frac{p}{q}\Big)\,-\,\frac{a}{b}\,
  \right|
\end{equation}
for every rational number $\frac{a}{b}$, except $\frac{p}{q}$,
$\gamma^{\pm}_{k-1}\big(\frac{p}{q}\big)$, and
$\gamma^{\pm}_{k}\big(\frac{p}{q}\big)$. The general strategy is the
same as in the previous section. We will use
Lemma~\ref{lem:nearest_to_companion} to eliminate all but a manageable
subset of candidates $\frac{a}{b}$ that we have to check. However,
instead of $G$-orbits as in Lemma~\ref{lem:nearest_to_Markov}, we will
consider orbits with respect to a larger group~$\widehat{G}$, of which
$G$ is a subgroup of index $2$:

\begin{definition}
  Let $\widehat{G}\subset\PSLTZ$ be the group generated by the group
  $G$ of deck transformations of $M$ and any lift of the hyperelliptic
  involution, e.g., $z\mapsto z+3$.
\end{definition}

\begin{remark}
  In the previous section, it was sufficient to consider the
  $G$-orbits of $g_{x}$, because the hyperelliptic involution only
  reverses the orientation of simple geodesics in $M$ with both ends
  in the cups. Hence, the $G$-orbit and the $\widehat{G}$-orbit of
  $g_{x}$ are equal if we disregard the orientation of geodesics.
\end{remark}

\begin{lemma}
  \label{lem:nearest_to_companion}
  If~$y$ is a companion of a Markov fraction and~$y^{*}$ is a best
  approximant of~$y$, then no geodesic in the $\widehat{G}$-orbit of
  $g_{y}$ separates~$g_{y}$ and~$y^{*}$.
\end{lemma}

Lemma~\ref{lem:nearest_to_companion} can be proved in the same way as
the analogous Lemma~\ref{lem:nearest_to_Markov} for Markov
fractions. Next, we will use it to prove a second geometric lemma:

\begin{lemma}
  \label{lem:nearest_to_companion2}
  If~$y$ is a companion of a Markov fraction $x$, and~$y^{*}$ is a best
  approximant of~$y$, then no geodesic in the $G$-orbit of
  $g_{x}$ separates~$g_{y}$ and~$y^{*}$.
\end{lemma}

\begin{proof}
  Let us assume that $y$ is a right companion of $x$, i.e.,
  \begin{equation}
    \label{eq:y_gamma}
    y\;=\;\gamma^{+}_{k}(x),\quad k\geq 2\,.
  \end{equation}
  The argument for left companions is analogous. The proof has two
  parts.

  \medskip\noindent%
  Part 1.\;\ We will first prove the following statement:
  \begin{quotation}\noindent
    For all real numbers $y'<x$ there is a geodesic in the
    $\widehat{G}$-orbit of $g_{y}$ that separates $g_{x}$ and $y'$.
  \end{quotation}
  To this end, we consider two cases separately:
  \medskip%
  \begin{compactitem}
  \item $y'\in\,(-\sminfty,\, y-3)$.\quad Let $\tau(z)=z-3$. Then
    $\tau\in\widehat{G}$, and because $y-3<x$, the vertical geodesic
    $\tau\cdot g_{y}$ separates $g_{x}$ and $y'$.

    \medskip
  \item $y'\in\,(\gamma^{+}_{k-1}(x)\,-\,3,\, x)$.\quad With $[T]\in G$ as in
    Section~\ref{sec:pf_companions_geom} and $\tau$ as above we have
    \begin{equation*}
      [T]^{-1}\cdot y \;=\; \gamma^{+}_{k-1}(x)
      \qquad\text{and}\qquad
      [T]^{-1}\cdot \sminfty \;=\; x+3.
    \end{equation*}
    The geodesic $\tau\circ[T]^{-1}\cdot g_{y}$ therefore has ideal
    endpoints $\gamma^{+}_{k-1}(x)-3$ and $x$. So it separates $g_{y}$
    and $y'$.
  \end{compactitem}
  \medskip%
  Since $\gamma^{+}_{k-1}(x)<y$, the two intervals cover
  $(-\sminfty,\,x)$, and hence we have shown the statement for all $y'<x$.

  \medskip\noindent%
  Part 2.\; Now to prove the lemma, suppose there is a $\sigma\in G$ such
  that $\sigma\cdot g_{x}$ separates $g_{y}$ and $y'$. We have to show
  that $y'$ is not a best approximant of $y$.

  To see this, let $\hat{\sigma}\in\widehat{G}$ be the element with
  \begin{equation*}
    \hat{\sigma}\cdot g_{x}\;=\;\sigma\cdot g_{x}
    \quad\text{and}\quad
    y' \;\in\; \hat{\sigma}\cdot \R_{<x}\,.
  \end{equation*}
  (Either $\hat{\sigma}=\sigma$, or $\sigma^{-1}\hat{\sigma}$ is the
  lift of the hyperelliptic involution on $M$ that maps the simple geodesic
  $g_{x}$ with both ends in the cusp to itself, but reversing its
  orientation.) Then $\hat{\sigma}^{-1}\cdot y'<x$, so by (1) there is
  a geodesic $g_{1}\in\widehat{G}\cdot g_{y}$ that
  separates $g_{x}$ and $\hat{\sigma}^{-1}\cdot y'$. Hence
  $\hat{\sigma}\cdot g_{1}$ separates $\hat{\sigma}\cdot g_{x}$ and
  $y'$. Since $\hat{\sigma}\cdot g_{x}$ separates $g_{y}$ and $y'$ by
  assumption, the geodesic $\hat{\sigma}\cdot g_{1}$ separates $g_{y}$
  and $y'$. By Lemma~\ref{lem:nearest_to_companion}, $y'$ is not a
  best approximant of $y$.
\end{proof}

To prove the theorem, let us consider a right companion $y$ as
in~\eqref{eq:y_gamma}. The arguments for a left companion are
completely analogous.

First, Lemma~\ref{lem:nearest_to_companion2} implies that the best
approximants of $y$ are among the rational ideal boundary points of
the domains $R_{x}$ and $L_{x}'$ (see Figure~\ref{fig:arc_orbit}).

Let us first consider the rational ideal boundary points of $R_{x}$,
which are
\begin{equation*}
  \gamma^{+}_{\ell}(x)\;=\;[T]^{\ell}\cdot \sminfty
  \quad\text{for}\quad
  \ell\;\in\;\Z
\end{equation*}
(see Section~\ref{sec:pf_companions_geom}). Of these,
$y=\gamma^{+}_{k}(x)$ and $\sminfty=\gamma^{+}_{0}(x)$ are by
definition not best approximants of $y$.

Next, we will use Lemma~\ref{lem:nearest_to_companion} to show that
$\gamma^{+}_{\ell}(x)$ is not a best approximant if $k<\ell$ or
$\ell<0$. To see this, note first that the ideal endpoints of the
geodesic $[T]^{m}\cdot g_{y}$ are
\begin{equation*}
  \gamma^{+}_{m}(x)
  \quad\text{and}\quad
  \gamma^{+}_{m+k}(x).
\end{equation*}
Hence, if
\begin{equation*}
  k\;\leq\; m\;<\;\ell\;<\;m+k
\end{equation*}
or
\begin{equation*}
  m\;<\;\ell\;<\;m+k\;\leq\;0
\end{equation*}
then $[T]^{m}\cdot g_{y}$ separates $g_{y}$ and
$\gamma^{+}_{\ell}(x)$. Thus we have eliminated all rational ideal
boundary points of $R_{x}$ as best approximants except
\begin{equation}
  \label{eq:Rx_candiates}
  x\;=\;\gamma^{+}_{1}(x),\quad
  \gamma^{+}_{2}(x),\;
  \ldots\quad
  \gamma^{+}_{k-1}(x)\,.
\end{equation}

To eliminate all rational ideal boundary points of $L_{x}'$, let
$\iota\in\widehat{G}$ be the lift of the hyperelliptic involution on $M$
that maps the geodesic $\varrho_{x}$ to itself, only reversing its
orientation. Then the rational ideal boundary points of $L_{x}'$ are
the numbers
\begin{equation*}
  \iota\cdot\gamma^{+}_{\ell}(x)\;=\;\iota\circ [T]^{\ell}\cdot \sminfty
  \quad\text{for}\quad
  \ell\;\in\;\Z
\end{equation*}
Thus, if
\begin{equation*}
  m\;<\;\ell\;<\;m+k
\end{equation*}
then the geodesic $\iota\circ [T]^{m}\circ g_{y}$ separates $g_{y}$
and $\iota\circ\gamma^{+}_{\ell}(x)$. By
Lemma~\ref{lem:nearest_to_companion}, this excludes all rational ideal
boundary points of $L_{x}'$ as best companions of $y$.

Now the only candidates for best approximant of $y$ are the
numbers~\eqref{eq:Rx_candiates}, and we know that $x$ and
$\gamma^{+}_{k-1}(x)$ have the same approximation quality. To prove that
they are indeed the best approximants, it remains to
show~\eqref{eq:companion_candidate_test} for
\begin{equation*}
  \frac{a}{b}\;=\;\gamma^{+}_{\ell}(x)
  \quad\text{with}\quad
  1\;<\;\ell\;<\;k-1\,.
\end{equation*}
If $k=2$, there is nothing to show. For $k>2$, we will take the
geometric point of view to simplify the calculations.

We have to show that among the Ford circles at the the rational
numbers~\eqref{eq:Rx_candiates}, the minimal signed distance to the
vertical geodesic $g_{y}$ is attained precisely for the first and the
last (see Lemma~\ref{lem:d_vert_horo}). Note that the Ford circles at
$\gamma^{+}_{\ell}(x)$ with $\ell\in\Z$ are equally spaced translates
along the translation axis $\varrho_{x}$. In particular, all these
horocycles have the same signed distance to $\varrho_{x}$. Moreover,
$1\leq\ell\leq k-1$ if and only if $g_{y}$ separates $\varrho_{x}$ and
the center $\gamma^{+}_{\ell}(x)$.

To simplify the calculations, we apply a hyperbolic isometry that maps
the axis~$\varrho_{x}$ to the upward oriented vertical geodesic
$g_{0}$, and the geodesic $g_{y}$ to the geodesic with ideal endpoints
$e^{-t_{0}}$ and $e^{t_{0}}$ for some positive real $t_{0}$ (see
Figure~\ref{fig:optimal_companion}).
\begin{figure}
  \centering
  \begin{tikzpicture}

  \draw (-1,0) -- (9,0);
  \draw (0,0) node[below] {\small$0$} -- (0,6);

  \draw (2,0) node[below] {\small$e^{-t_{0}}$} arc (180:0:3) node[below] {\small$e^{t_{0}}$};

  \fill [black] (0,0) circle (1pt);
  \fill [black] (2,0) circle (1pt);
  \fill [black] (2.828,0) node[below] {\small$e^{t}$} circle (1pt);
  \fill [black] (4,0) node[below] {\small$1$} circle (1pt);
  \fill [black] (8,0) circle (1pt);

  \fill [black] (0,4) node[left] {\small$i$} circle (1pt);

  \draw[ultra thin] (2.818, 1.414) circle (1.414);
  \draw (1.8,3.1) node {\small$\displaystyle h\big(e^{\frac{t}{2}},e^{\frac{-t}{2}}\big)$};

  \draw[ultra thin]  (2.828,0) arc (0:110.7:.9866);
  \draw [thick] ($(1.8418,0)+({.9866*cos(71.5)},{.9866*sin(71.5)})$) arc (71.5:110.7:.9866);
  \draw (1.8, 1.2) node {$d$};
\end{tikzpicture}

  \caption{The signed distance $d$ between the geodesic with ideal
    endpoints~$e^{\pm t_{0}}$ and the horocycle
    $h(e^{\frac{t}{2}},e^{\frac{-t}{2}})$}
  \label{fig:optimal_companion}
\end{figure}
Then the images of the Ford circles at $\gamma^{+}_{\ell}(x)$ are
among the $1$-parameter family of translates of some horocycle
$h(a,a)$ along the axis $g_{0}$. To simplify the calculations further,
we will consider the translates of $h(1,1)$ instead. This scales the
horocycles evenly, so all distances change by the same additive
constant. That is, we consider the family of horocycles $h(p(t),q(t))$
with
\begin{equation*}
  \begin{pmatrix}
    p(t)\\q(t)
  \end{pmatrix}
  \;=\;
  \begin{pmatrix}
    e^{\frac{t}{2}} & 0 \\
    0 & e^{-\frac{t}{2}}
  \end{pmatrix}
  \,
  \begin{pmatrix}
    1 \\ 1
  \end{pmatrix}
  \;=\;
  \begin{pmatrix}
    e^{\frac{t}{2}} \\ e^{-\frac{t}{2}}
  \end{pmatrix}
\end{equation*}
(see Section~\ref{sec:overview_hyperbolic} and~\cite[p.~665]{Fock07},
\cite[Ch.~5]{Springborn_Markov}). Among these, the scaled images of
the Ford circles at at $\gamma^{+}_{\ell}(x)$ are the horocycles
$h(e^{\frac{t}{2}},e^{-\frac{t}{2}})$ with
\begin{equation}
  \label{eq:t_of_ell}
  t\;=\;\frac{t_{0}}{k}\;(\,2\,\ell\,-\,k\,)\,.
\end{equation}
(To see this, note that for integer values of $\ell$ between $0$ and
$k$, the $t$-values are evenly spaced reals between~$-t_{0}$
and~$t_{0}$.)

The signed distance between the horocycle
$h(e^{\frac{t}{2}},e^{-\frac{t}{2}})$ and the geodesic with ideal
endpoints~$e^{\pm t_{0}}$ is
\begin{equation*}
  d\;=\;\log\,
  \left|\,
    f\big(e^{\frac{t}{2}},\,e^{\frac{-t}{2}}\big)\,
\right|\,,
\end{equation*}
where $f$ is a quadratic form on $\R^{2}$ with determinant $-1$ that is zero for
$\frac{p}{q}=e^{\pm t_{0}}$ (see
\cite[Ch.~10]{Springborn_Markov}). This determines $f$ up to sign and
we may choose
\begin{equation*}
  f(p,q)\;=\;
  \frac{1}{\sinh t_{0}}
  \big(
  p^{2}\,-\,2\,(\cosh t_{0})\,p\,q\,+\,q^{2})
  \big)\,.
\end{equation*}
For $t\in (-t_{0},t_{0})$ we obtain
\begin{equation}
  \label{eq:d_companions}
  e^{d}\;=\;
  \frac{2}{\sinh t_{0}}\,
  \big(\cosh t_{0}\,-\,\cosh t\big)\,.
\end{equation}
This is a concave even function of $t$ with limit zero as
$t\rightarrow \pm t_{0}$. Hence, for $t$ as in~\eqref{eq:t_of_ell}
with $\ell\in\Z$ and $1\leq\ell\leq k-1$, the signed distance $d$
attains its minimum precisely for $\ell=1$ and $\ell=k-1$. This
completes the proof of Theorem~\ref{thm:companion_approximation}.

\subsection{Empty intervals around Markov fractions}
\label{sec:Ipq}

This section is about the largest intervals $I_{x}$ around each Markov
fraction~$x$ that contain no other Markov fractions (see
Section~\ref{sec:companions}). At the same time, $I_{x}$ is the
smallest closed interval that contains all companions of $x$. We will
prove Lemmas~\ref{lem:Ipq} and~\ref{lem:Ipq_cover}, which will be used
in the following section to complete the proof of the classification
Theorem~\ref{thm:classify}.

From the geometric point of view, Lemmas~\ref{lem:Ipq}
and~\ref{lem:Ipq_cover} are well known, and even in greater
generality: for arbitrary hyperbolic tori with one cusp. This is the
origin of McShane's remarkable identity for the lengths of simple
closed geodesics~\cite{mcshane91} (see also Remark~\ref{rem:mcshane}),
which has been reproved~\cite{bowditch96,goodman-strauss07,schmidt08}
and generalized~\cite{akiyoshi04,mcshane98,mcshane04} in various
ways.  Nevertheless, it seems worthwhile to provide independent proofs
in the current context.

\begin{lemma}
  \label{lem:Ipq}
  For every Markov fraction $x=\frac{p}{q}$, the interval $I_{x}$
  defined by equation~\eqref{eq:Ipq} is the smallest closed interval
  that contains all companions of $x$. It contains no Markov fractions
  other than $x$, nor any of their companions. But the irrational
  endpoints of~$I_{x}$ are accumulation points of the set of Markov
  fractions.
\end{lemma}

\begin{proof}[Proof of Lemma~\ref{lem:Ipq}]
  By equations~\eqref{eq:Tk_lim} and~\eqref{eq:companion_geometric},
  the right endpoint $\xi_{+}$ of $I_{x}$ is the limit of the
  increasing sequence $(\gamma^{+}_{k}(x))_{k\geq 2}$ of right
  companions of $x$. Analogously or by symmetry, the left endpoint is
  the limit of the decreasing sequence of left companions. Thus no
  smaller closed interval contains all companions.

  The rational ideal boundary points of the domain $L_{x}'$ are Markov
  numbers, and they accumulate at the right endpoint of $I_{x}$ (see
  Section~\ref{sec:pr_Markov_approximants} and
  Figures~\ref{fig:companions_geometric}
  and~\ref{fig:arc_orbit}). Similarly, the rational ideal boundary
  points of $R_{x}'$ are Markov numbers accumulating at the left
  endpoint of $I_{x}$. As accumulations points of Markov numbers, the
  endpoints of $I_{x}$ are irrational (see
  Section~\ref{sec:path_limit}; in fact, it is well known that the
  endpoints are Markov irrationals).

  For every interior point $y$ of $I_{x}$ except $x$, the geodesic
  $\pi(g_{y})$ in $M$ has self-intersections. To see this if $y>x$,
  note that $g_{y}$ intersects $[T]\cdot g_{y}$ (see
  Section~\ref{sec:pf_companions_geom}). The same holds for $y<x$ by
  an analogous argument or by symmetry. Since the endpoints of
  $I_{x}$ are not even rational, this implies that $x$ is the only
  Markov fraction in $I_{x}$.

  In particular, this means that no other Markov fraction is between
  $x$ and one of its companions. Since this is true for all Markov
  fractions, and they accumulate at the endpoints of $I_{x}$, this
  implies that $I_{x}$ does not contain any companions of other Markov
  fractions than $x$.
\end{proof}

\begin{corollary}
  \label{cor:I_x_disjoint}
  For different Markov fractions $x\not=x'$, the interiors 
  $\mathring{I}_{x}$ and $\mathring{I}_{x'}$ are disjoint.
\end{corollary}

\begin{lemma}
  \label{lem:Ipq_cover}
  As $x$ ranges over the set of Markov fractions, the intervals
  $I_{x}$ cover the set $\Q$ of rational numbers.
\end{lemma}

\begin{proof}
  Since the intervals obviously cover the set of Markov fractions, let
  us assume that $y\in\Q$ is not a Markov fraction. We have to find a
  Markov fraction~$x$ for which $y\in I_{x}$.

  To this end, define the infinite sequence of centered Markov triples
  \begin{equation*}
    m_{k}\;=\;
    \left(\,
      \frac{p_{k,1}}{q_{k,1}},\;
      \frac{p_{k,2}}{q_{k,2}},\;
      \frac{p_{k,3}}{q_{k,3}}\,
    \right)
  \end{equation*}
  with
  \begin{equation}
    \label{eq:mk_nesting}
    \frac{p_{k,1}}{q_{k,1}}
    \;<\; y \;<\;
    \frac{p_{k,3}}{q_{k,3}}
  \end{equation}
  recursively by
  \medskip%
  \begin{compactitem}
  \item $m_{0}
    \;=\;
    \left(\,
      \lfloor y\rfloor,\;
      \lfloor y\rfloor+\tfrac{1}{2},\;
      \lfloor y\rfloor+1\,
    \right)$, where $\lfloor y\rfloor$ denotes the largest integer not
    larger than $y$, and
    \medskip
  \item for $k>0$, the triple $m_{k}$ is the child
    satisfying~\eqref{eq:mk_nesting} among the two children of
    $m_{k-1}$ in the tree of centered rational Markov triples (see
    Lemma~\ref{lem:children_parents}).
  \end{compactitem}

  \medskip\noindent{}%
  First consider the following two cases:
  \begin{compactenum}[1.]
  \medskip
  \item The triple $m_{k}$ is the left child of $m_{k-1}$ for all but
    finitely many $k$.

    \medskip%
    That is, the sequence of Markov fractions
    $\frac{p_{k,1}}{q_{k,1}}$ eventually becomes constant, say
    \begin{equation*}
      \frac{p_{k,1}}{q_{k,1}}\;=\;x
      \quad\text{for}\quad k\;>\;k_{0}\,,
    \end{equation*}
    while $\frac{p_{k,2}}{q_{k,2}}$ and $\frac{p_{k,3}}{q_{k,3}}$ are
    consecutive ideal boundary points of $L_{x}'$ with
    \begin{equation*}
      \lim\,\frac{p_{k,2}}{q_{k,2}}
      \;=\;
      \lim\,\frac{p_{k,3}}{q_{k,3}}
      \;=\;\xi_{+}
    \end{equation*}
    (see Section~\ref{sec:pr_Markov_approximants} and
    Figures~\ref{fig:companions_geometric}
    and~\ref{fig:arc_orbit}). Hence $y\in I_{x}$

    \medskip
  \item The triple $m_{k}$ is the right child of $m_{k-1}$ for all but
    finitely many $k$.

    \medskip%
    In this case,
    \begin{equation*}
      \frac{p_{k,3}}{q_{k,3}}\;=\;x
      \quad\text{for}\quad k\;>\;k_{0}\,,
    \end{equation*}
    and like in case 1 we see that $y\in I_{x}$.
  \end{compactenum}

\medskip\noindent%
We will complete the proof of the lemma by showing that the third case
does not occur:
  \begin{compactenum}[1.]

    \medskip{}
  \item[3.] Infinitely many of the triples $m_{k}$ are left children
    of $m_{k-1}$ and infinitely many are right children.
  \end{compactenum}

  \medskip\noindent%
  For such a sequence of rational Markov triples, we will show that
  \begin{equation}
    \label{eq:markov_squeeze}
    \lim
    \left(
      \frac{p_{k,3}}{q_{k,3}}
      \,-\,
      \frac{p_{k,1}}{q_{k,1}}
    \right)
    \;=\;
    0\,.
  \end{equation}
  Then condition~\eqref{eq:mk_nesting} would imply
  \begin{equation*}
    \lim\,\frac{p_{k,1}}{q_{k,1}}
    \;=\;
    y
    \;=\;
    \lim\,\frac{p_{k,3}}{q_{k,3}}\,.
  \end{equation*}
  This would contradict $y$ being rational, because accumulation
  points of the set of Markov fractions are irrational (see
  Section~\ref{sec:path_limit}).

  It remains to show~\eqref{eq:markov_squeeze} in case 3. To this end,
  note that for any rational Markov triple 
  $\big(\frac{p_{1}}{q_{1}},\frac{p_{2}}{q_{2}},\frac{p_{3}}{q_{3}}\big)$
  we have
  \begin{equation*}
    \begin{split}
      \frac{p_{3}}{q_{3}}\,-\,\frac{p_{1}}{q_{1}}
      \quad&=\quad
      \left(
        \frac{p_{3}}{q_{3}} \,-\, \frac{p_{2}}{q_{2}}
      \right)
      \,+\,
      \left(
        \frac{p_{2}}{q_{2}} \,-\, \frac{p_{1}}{q_{1}}
      \right)\\
      &\underset{\makebox[0pt]{\scriptsize\eqref{eq:p2p3}}}{\overset{\makebox[0pt]{\scriptsize\eqref{eq:p1p2}}}{=}}\quad
      \frac{q_{1}}{q_{2}\,q_{3}}\,+\,\frac{q_{3}}{q_{1}\,q_{2}}\\
      &\overset{\makebox[0pt]{\scriptsize\eqref{eq:Markov_q}}}{=}\quad
      3\,-\,
      \frac{q_{2}}{q_{3}\,q_{1}}\,,
    \end{split}
  \end{equation*}
  so it is enough to show
  \begin{equation}
    \label{eq:lim_q2_q3q1}
    \lim\,\frac{q_{k,2}}{q_{k,3}\,q_{k,1}}\;=\;3.
  \end{equation}
  To see this, we will show that for any centered rational Markov
  triple
  $\big(\frac{p_{1}}{q_{1}},\frac{p_{2}}{q_{2}},\frac{p_{3}}{q_{3}}\big)$
  we have
  \begin{equation}
    \label{eq:q2_q3q1_ineq}
    \frac{q_{2}}{q_{3}\,q_{1}}\;\geq\;
    \frac{3}{
      1
      \,+\,
      \frac{1}{q_{1}^{2}}
      \,+\,
      \frac{1}{q_{3}^{2}}
    }\,.
  \end{equation}
  Since both sequences $(q_{k,1})$ and $(q_{k,3})$ are unbounded in
  case 3, equation~\eqref{eq:lim_q2_q3q1} follows.
  
  Finally, to prove the estimate~\eqref{eq:q2_q3q1_ineq} for a
  centered rational Markov triple, note that~\eqref{eq:centered}
  implies
  \begin{equation*}
    \frac{q_{2}}{q_{3}\,q_{1}}
    \;\geq\;
    \max\,
    \left\{\,
      \frac{q_{1}}{q_{2}\,q_{3}}\,,\,
      \frac{q_{3}}{q_{1}\,q_{2}}\,
    \right\}\,,
  \end{equation*}
  and with equation~\eqref{eq:Markov_q},
  \begin{equation*}
    \frac{q_{2}}{q_{3}\,q_{1}}\;\geq\;1\,,
  \end{equation*}
  Hence $\frac{q_{2}}{q_{1}}\geq q_{3}$ and therefore
  \begin{equation*}
    \frac{q_{1}}{q_{2}\,q_{3}}
    \;=\;
    \frac{q_{2}}{q_{3}\,q_{1}}
    \cdot \frac{q_{1}^{2}}{q_{2}^{2}}
    \;\leq\;
    \frac{q_{2}}{q_{3}\,q_{1}}\cdot\frac{1}{q_{3}^{2}}\,.
  \end{equation*}
  In the same way we get $\frac{q_{2}}{q_{3}}\geq q_{1}$ and
  \begin{equation*}
    \frac{q_{3}}{q_{1}\,q_{2}}
    \;\leq\;
    \frac{q_{2}}{q_{3}\,q_{1}}\cdot\frac{1}{q_{1}^{2}}\,.
  \end{equation*}
  Now equation~\eqref{eq:Markov_q} yields
  \begin{equation*}
    3\;\leq\;
    \frac{q_{2}}{q_{3}\,q_{1}}
    \,
    \left(
      \frac{1}{q_{3}^{2}}
      \,+\,1\,+\,
      \frac{1}{q_{1}^{2}}
    \right)
  \end{equation*}
  and hence~\eqref{eq:q2_q3q1_ineq}. This completes the proof of
  Lemma~\ref{lem:Ipq_cover}.
\end{proof}

\begin{remark}
  \label{rem:mcshane}
  McShane's identity~\cite{mcshane91} says that
  \begin{equation}
    \label{eq:mcshane}
    \sum_{\gamma}\frac{1}{1\,+\,e^{|\gamma|}}
    \;=\;
    \frac{1}{2}\,,
  \end{equation}
  where the sum is taken over all simple closed geodesics in some
  hyperbolic torus with one cusp, and $|\gamma|$ denotes the length of
  the geodesic. For the modular torus $M$, this identity can be
  derived as follows. The unoriented closed geodesics in~$M$ are in
  one-to-one correspondence with the
  $(\operatorname{mod} 3)$-equivalence classes of Markov fractions.
  The length of the closed geodesic $\eta_{x}$ and the
  length of the interval $I_{x}$ are related by
  \begin{equation}
    \label{eq:eta_I_lengths}
    \frac{1}{1\,+\,e^{|\gamma|}}
    \;=\;
    \tfrac{1}{6}\;|\,I_{x}\,|
  \end{equation}
  (see Remark~\ref{rem:len_closed_geodesic}). On the other hand,
  Corollary~\ref{cor:I_x_disjoint} and Lemma~\ref{lem:Ipq_cover} imply
  \begin{equation}
    \label{eq:mcshane_I}
    2\,\sum_{x}|\,I_{x}\,|\;=\;6\,
  \end{equation}
  which yields McShane's identity.
\end{remark}

\subsection{Proof of Theorem~\ref{thm:classify} (Classification)}
\label{sec:pf_classify}

Since we have already proved
Theorems~\ref{thm:Markov_best_approximants}
and~\ref{thm:companion_approximation} in the two previous sections, we
only have to show that $C(y)<\frac{1}{3}$ if $y$ is a rational number
that is neither a Markov fraction nor a companion of a Markov
fraction.

By Lemma~\ref{lem:Ipq_cover} there is a Markov fraction $x$ such that
$y\in I_{x}$. Since we assume that $y$ is not a Markov fraction or a
companion, Lemma~\ref{lem:Ipq} implies that~$y$ lies between two
companions of $x$ or between $x$ and one of its first companions. In
other words, there is an integer $k\geq 1$ such that
\begin{equation*}
  \gamma^{-}_{k+1}(x)\;<\;y\;<\;\gamma^{-}_{k}(x)
  \qquad\text{or}\qquad
  \gamma^{+}_{k}(x)\;<\;y\;<\;\gamma^{+}_{k+1}(x)\,.
\end{equation*}
We will show that there is not enough space between the Ford circles
at $\gamma^{\pm}_{k}(x)$ and $\gamma^{\pm}_{k+1}(x)$ such that a
vertical geodesic could pass between them and stay at least a signed
distance of \,$\log\frac{2}{3}$\, away from both. With
Corollary~\ref{cor:C_geometric}, this shows that $C(y)<\frac{1}{3}$.

We already know that
\begin{equation}
  \label{eq:companion_gap2}
  \big|\,
    \gamma^{\pm}_{k}(x)
    \,-\,
    \gamma^{\pm}_{k+1}(x)\,
  \big|
  \;=\;
  \frac{1}{q\,u_{k}\,u_{k+1}}\,,
\end{equation}
and since the denominator a companion $\gamma^{\pm}_{k}(x)$ is
$qu_{k}$, the euclidean radius of its Ford circle is
\begin{equation*}
\frac{1}{2\,q^{2}\,u_{k}^{2}}
\end{equation*}
(see equation~\eqref{eq:companion_gap},
Lemma~\ref{lem:companion_reduced}, and
Section~\ref{sec:overview_hyperbolic}). A necessary condition for
$C(y)\geq\frac{1}{3}$ is that the vertical geodesic $g_{y}$ fits
between the Ford circles at $\gamma^{\pm}_{k}(x)$ and
$\gamma^{\pm}_{k+1}(x)$, scaled by $\frac{2}{3}$ in the euclidean
metric (see Corollary~\ref{cor:C_geometric} and
Figure~\ref{fig:ford}). Therefore, to prove that $C(y)<\frac{1}{3}$, it
suffices to show that
\begin{equation}
  \label{eq:gap_too_small}
  \big|\,
  \gamma^{\pm}_{k+1}(x)
  \,-\,
  \gamma^{\pm}_{k}(x)\,
  \big|
  \;<\;
  \frac{1}{3\,q^{2}\,u_{k+1}^{2}}
  \,+\,
  \frac{1}{3\,q^{2}\,u_{k}^{2}}\,.
\end{equation}
And indeed we have
\begin{equation*}
  \begin{split}
    \frac{1}{3\,q^{2}\,u_{k+1}^{2}} \,+\,
    \frac{1}{3\,q^{2}\,u_{k}^{2}} \;\;&=\;\;
    \frac{u_{k}^{2}\,+\,u_{k+1}^{2}}{3\,q^{2}\,u_{k}^{2}\,u_{k+1}^{2}}\\
    &\overset{\makebox[0pt]{\scriptsize\eqref{eq:u_conserved}}}{=}\;\;
    \frac{3\,q\,u_{k}\,u_{k+1}\,+\,1}{3\,q^{2}\,u_{k}^{2}\,u_{k+1}^{2}}\\
    &=\;\;
    \frac{1}{q\,u_{k}\,u_{k+1}}
    \,+\,\frac{1}{3\,q^{2}\,u_{k}^{2}\,u_{k+1}^{2}}\\
    &\overset{\makebox[0pt]{\scriptsize\eqref{eq:companion_gap2}}}{>}\;\;
    \big|\,
    \gamma^{\pm}_{k}(x)
    \,-\,
    \gamma^{\pm}_{k+1}(x)\,
    \big|\,.
  \end{split}
\end{equation*}

\vspace{-1.5\baselineskip}\qed

\section{Coda: Triangle paths}
\label{sec:triangle_paths}

\subsection{Triangle paths for Markov fractions}
\label{sec:triangle_paths_markov}

The Markov fraction $\mu_{\frac{n}{m}}$ defined in
Section~\ref{sec:labeling} can be computed from its
label~$\frac{n}{m}$ by a Fibonacci-like recursion with Farey addition
along triangle paths (see Figure~\ref{fig:snake_Markov}).
\begin{figure}[p]
  \input{snake_markov_fig}  
  \caption[]{Triangle paths of the Markov fractions $\mu_{\frac{n}{m}}$
    with $1\leq n<m\leq 5$.}
  \label{fig:snake_Markov}
\end{figure}
This construction, which is described more precisely in the following
theorem, extends Propp's snake graph construction for Markov
numbers~\cite{Propp}.

\begin{theorem}[Triangle paths for Markov fractions]
  \label{thm:snake_markov}
  Let $m\geq 0$ and $n\geq 0$ be coprime integers, let
  $z_{m,n}=m+n\,\omega$ with $\omega=\frac{1}{2}\,(1+i\sqrt{3})$,
  and let~$\alpha$ be the oriented straight line segment from~$0$ to
  $z_{m,n}$. Let $\tau_{0}\;=\;[0,\;1,\;\omega]$, and unless $(m,n)$
  equals $(1,0)$ or $(0,1)$, let
  \begin{equation*}
    \tau_{1}\;=\;[1,\;\omega,\;1+\omega],
    \quad\ldots,\quad
    \tau_{N}=[z_{m,n}-1,\;z_{m,n}-\omega,\; z_{m,n}]
  \end{equation*}
  be the finite sequence of triangles in the Eisenstein lattice
  $\Z+\omega\,\Z$ that~$\alpha$ intersects after
  $\tau_{0}$. Recursively label the vertices of these triangles as
  follows:
  \begin{compactitem}
  \item Label the vertices $0$, $1$, and $\omega$ of $\tau_{0}$ with
    $\frac{1}{0}$, $\frac{0}{1}$, and $\frac{1}{1}$, respectively.
  \item If the vertices of $\tau_{0},\ldots,\tau_{j-1}$ are already
    labeled, label the remaining vertex of triangle $\tau_{j}$ with
    $\frac{p_{1}+p_{2}}{q_{1}+q_{2}}$, where $\frac{p_{1}}{q_{1}}$
    and $\frac{p_{2}}{q_{2}}$ are the labels at the vertices shared
    with triangle $\tau_{j-1}$.
  \end{compactitem}

  \medskip\noindent%
  Then the label at $z_{m,n}$ is the Markov fraction $\mu_{\frac{n}{m}}$.
\end{theorem}

The companions of a Markov fraction can be obtained by a similar
construction, which will be described in the following section. Then,
in Section~\ref{sec:proof_triangle_paths}, both constructions will be
derived by following geodesics in the Farey triangulation.

\begin{remark}[Triangle paths and hyperbolic geodesics]
  In Figure~\ref{fig:snake_Markov}, the point $0$
  (labeled~$\frac{1}{0}$) is connected to the points $z_{m,n}$
  (labeled with the Markov fractions~$\mu_{\frac{n}{m}}$) not by
  straight line segments as described in
  Theorem~\ref{thm:snake_markov}, but by polygonal curves that are
  determined as follows. Define a hyperbolic metric on each triangle
  strip by equipping each triangle with the hyperbolic Beltrami--Klein
  metric defined by its circumcircle~\cite[Ch.~5.1]{BPS}. In this
  metric, the polygonal curves shown in Figures~\ref{fig:snake_Markov}
  and~\ref{fig:snake_companion} are hyperbolic geodesics connecting
  their ideal endpoints. We use this construction because it works
  equally well for companions (see Figure~\ref{fig:snake_companion}),
  whereas the straight line segments would have vertices in their
  interior. More importantly, Markov fractions and their companions
  correspond to certain geodesics in the modular torus with both ends
  in the cusp, and the polygonal curves shown in
  Figures~\ref{fig:snake_Markov} and~\ref{fig:snake_companion} are
  natural representations of these geodesics in the following sense:
  The canonical ideal triangulation of the modular torus induces a
  euclidean metric on it~\cite{Epstein_Penner}, and if one uses this
  euclidean metric to lift the hyperbolic geodesics to the euclidean
  plane, one obtains the polygonal curves shown in the figures.
\end{remark}

\subsection{Triangle paths for companions}

The left and right companions of a Markov fraction can also be found
by a Fibonacci-like recursion along triangle paths (see
Figure~\ref{fig:snake_companion}).
\begin{figure}[t]
  \input{snake_companion_fig}
  \caption[]{(a) Triangle paths for $\mu_{\frac{1}{2}}=\frac{2}{5}$ and the
    first and second right companions
    $\gamma^{+}_{2}(\frac{2}{5})=\frac{31}{75}$ and
    $\gamma^{+}_{3}(\frac{2}{5})=\frac{463}{1120}$.
    (b) Same for the first and second left companions
    $\gamma^{-}_{2}(\frac{2}{5})=\frac{29}{75}$ and
    $\gamma^{-}_{3}(\frac{2}{5})=\frac{433}{1120}$.}
  \label{fig:snake_companion}
\end{figure}
Roughly speaking, while the Markov fractions are obtained from
triangle paths of primitive lattice vectors, the companions are
obtained from triangle paths of imprimitive lattice vectors. To avoid
lattice points in their interior, the paths are bent slightly to
the the left or to the right, which leads to right or left companions,
respectively. More precisely:

\begin{theorem}[Triangle paths for companions]
  \label{thm:snake_companion}
  With $m$, $n$, $\omega$, $z_{m,n}$\,, and $\alpha$ as in
  Theorem~\ref{thm:snake_markov}, let $k$ be an integer $\geq 2$, so
  that the oriented line segment $k\alpha$ from~$0$ to
  $k\,z_{m,n}=z_{km,kn}$ contains $k-1$ vertices of the Eisenstein
  lattice in its interior. Bend this line segment slightly to the left
  or to the right while keeping the endpoints fixed to obtain an
  oriented curve $\beta^{+}$ or $\beta^{-}$, respectively, that avoids
  lattice points in the interior. Or, to be more precise and more
  definite, let $\beta^{\pm}$ be the curve
  \begin{equation*}
    \beta^{\pm}:\;[0,1]\rightarrow\C,\quad
    \beta^{\pm}(t)\;=\;(t\,\pm\,i\,\epsilon\,t\,(1-t))\,k\,z_{m,n},
  \end{equation*}
  where $\epsilon>0$ is chosen small enough so that there are no
  lattice points between the line segment~$k\alpha$ and $\beta^{\pm}$
  or in the interior of $\beta^{\pm}$. As in
  Theorem~\ref{thm:snake_markov}, label the vertices $0$, $1$, and
  $\omega$ with $\frac{1}{0}$, $\frac{0}{1}$, and $\frac{1}{1}$,
  respectively, and label the vertices of the triangles
  $ \tilde\tau^{\pm}_{0},\,\ldots,\,\tilde\tau^{\pm}_{N}$ crossed by
  $\beta^{\pm}$ recursively, so that in each triangle of the sequence
  the as yet unlabeled vertex is labeled with the Farey sum of the
  already labeled vertices.

  \medskip\noindent%
  Then the label at $k\,z_{m,n}$ is $\gamma^{\pm}_{k}\big(\frac{p}{q}\big)$,
  the $(k-1)$st left or right companion of $\frac{p}{q}=\mu_{\frac{n}{m}}$.
\end{theorem}

Like the analogous construction for Markov fractions (see
Theorem~\ref{thm:snake_markov}), this is derived by following
geodesics in the Farey triangulation (see
Section~\ref{sec:proof_triangle_paths}).

\subsection{Proof of Theorems~\ref{thm:snake_markov}
  and~\ref{thm:snake_companion}}
\label{sec:proof_triangle_paths}

The Epstein--Penner convex hull construction~\cite{Epstein_Penner}
provides the modular torus~$M$ with a euclidean metric, with respect
to which the ideal hyperbolic triangles of the canonical triangulation
are equilateral euclidean triangles. We may assume that the developing
map $D:H^{2}\rightarrow\C$ maps the ideal triangle $\infty,0,1$ to the
euclidean triangle $0,1,\omega$ with $\omega$ as in
Theorem~\ref{thm:snake_markov}. The image of $D$ is then the
complement $\C\setminus\Gamma$ of the Eisenstein lattice
$\Gamma\;=\;\Z\,+\,\omega\Z$. The canonical projection $\pi$ factors
through $\C\setminus\Gamma$, giving rise to the projection
$\bar{\pi}$:
\begin{equation*}
  \begin{tikzcd}
    H^2 \arrow[r, "D"]\arrow[d, "\pi"] & \C\setminus\Gamma \arrow[dl, "\bar{\pi}"]\\
    M & 
  \end{tikzcd}
\end{equation*}

If $m$ and $n$ are coprime integers, then the projection $\bar{\pi}$
maps the oriented line segment from $0$ to $z_{m,n}=m+n\omega$ to a
simple closed curve in $M$. There is a unique homotopic simple
geodesic in $M$ with both ends in the cusp (under homotopy with ends
in the cusp), and all simple geodesics with both ends in the cusp can
be obtained in this way. If $m$ and $n$ are both nonnegative, the
geodesic in $M$ lifts to a vertical geodesic $g_{x}$ in $H^{2}$,
oriented from $\sminfty$ to $0$, where $x$ is a Markov fraction
between $0$ and $1$. As the geodesic $g_{x}$ traverses finitely many
triangles of the Farey triangulation, the new vertices are obtained by
Farey addition. But the developing map $D$ maps the triangles of the
Farey triangulation that $g_{y}$ traverses to the regular euclidean
triangles that the line segment from $0$ to $z_{m,n}$
traverses. Hence, the Markov fraction $x$ can also be obtained by
Farey addition along the euclidean triangle strip. This proves
Theorem~\ref{thm:snake_markov}.

To prove Theorem~\ref{thm:snake_companion} in a similar fashion, note
that $\bar{\pi}$ projects the deformed straight line segments to
curves in $M$ with both ends in the cusp that cross neither the
respective simple geodesic with both ends in the cusp nor the
corresponding simple closed geodesic.

\paragraph{Acknowledgment.}
I would like to thank Gergely Harcos, who graciously made me aware of
the previous results by Flahive~\cite{gbur78} and
Gurwood~\cite{gurwood76}.

\smallskip\noindent%
This research was funded by the Deutsche Forschungsgemeinschaft (DFG - German Research Foundation) - Project-ID 195170736 - TRR109

\begingroup
\let\OLDthebibliography\thebibliography
\renewcommand\thebibliography[1]{
  \OLDthebibliography{#1}
  \setlength{\parskip}{0pt}
  \setlength{\itemsep}{0.5ex plus 0.2ex}
}
\bibliographystyle{abbrv}
\bibliography{fractions}
\endgroup

\vspace{\baselineskip}
\noindent%
Boris Springborn\ \;\nolinkurl{<boris.springborn@tu-berlin.de>}\\[\baselineskip]
Technische Universit\"at Berlin\\
Institut f\"ur Mathematik, MA 8-3\\
Str.~des 17.~Juni 136\\
10623 Berlin, Germany

\end{document}